\documentclass[oneside, a4paper,11pt,reqno]{amsart}
\usepackage{yfonts}

\RequirePackage[colorlinks,citecolor=blue,urlcolor=blue]{hyperref}

\makeatletter
\def\timenow{\@tempcnta\time
  \@tempcntb\@tempcnta
  \divide\@tempcntb60
  \ifnum10>\@tempcntb0\fi\number\@tempcntb
  \multiply\@tempcntb60
  \advance\@tempcnta-\@tempcntb
  :\ifnum10>\@tempcnta0\fi\number\@tempcnta}
\makeatother

\usepackage[ddmmyyyy,hhmmss]{datetime}
\usepackage{color}

\usepackage{euscript}

\usepackage[applemac]{inputenc}
\usepackage{amssymb,amsmath,amsthm,dsfont}
\usepackage{graphics,graphicx}
\usepackage[T1]{fontenc}
\usepackage{color}
\usepackage{epsfig,psfrag}

\usepackage{enumerate}  

\newtheorem{theo}{Theorem}[section]
\newtheorem{prop}[theo]{Proposition}

\newtheorem{lemma}[theo]{Lemma}

\newtheorem{cor}[theo]{Corollary}

\newtheorem{remark}[theo]{Remark}

\def\E{\mathbb{E}}

\def\R{\mathbb{R}}
\def\dd{\textnormal{d}}

\def\0{\textbf{0}}

\def\Ms{\mathcal{M}}
\def\Es{\mathcal{E}}

\newcommand*{\affmark}[1][*]{\textsuperscript{#1}}

\title[Large deviation principle for the absorption time of the Beta-coalescent]{Large deviation principle for the absorption time of the Beta-coalescent via integral functionals}

\author{Gr\'egoire V\'echambre\affmark[1]}

\address{\affmark[1]State Key Laboratory of Mathematical Sciences, Academy of Mathematics and Systems Science, Chinese Academy of Sciences, Beijing 100190, China}
\email{vechambre@amss.ac.cn}

\begin{document} 

\maketitle

\begin{abstract}
We study some aspects of the absorption time of the Beta$(a,b)$-Coalescent starting with $n$ blocks. More precisely, when $a>1$, the absorption time is known to converge to infinity as $n$ goes to infinity, and we prove that it satisfies a large deviation principle. When $a \in (0,1)$, it is known that the coalescent comes down from infinity, and we derive bounds for the convergence in the Kolmogorov distance of the distribution of the absorption time as $n$ goes to infinity. To prove our results we introduce a method, inspired from statistical mechanics, that allows to infer the asymptotic behavior of the Laplace transforms of some integral functionals of the Beta-coalescent as the initial number of blocks $n$ goes to infinity. As a by-product of our proofs we also obtain estimates for the record probabilities of the Beta-coalescent. 
\end{abstract}

{\footnotesize
	
	}

\section{Introduction} \label{intro}

For a finite and non-zero measure $\Lambda$ on $[0,1]$ (we denote this by $\Lambda\in\Ms_f([0,1])$), the $\Lambda$-coalescent starting with $n$ blocks is a Markov process $(\Pi^n_t)_{t\geq 0}$ on partitions of $\{1,\ldots,n\}$ introduced by \cite{pitman1999} and \cite{sagitov1999}. This process has the following dynamics: if there are currently $p$ blocks in the partition, for $k \in \{2,\ldots,p\}$, any $k$ of the blocks merge into one with rate 
\begin{align}
\lambda_{p,k}(\Lambda):=\int_{[0,1]} r^{k-2}(1-r)^{p-k} \Lambda(dr). \label{coalrates}
\end{align}
The $\Lambda$-coalescent is an exchangeable coalescent that generalizes the classical Kingman coalescent (corresponding to the case $\Lambda=\delta_0$) by allowing multiple mergers instead of only binary mergers. It can also be seen as a particular case of coalescents with simultaneous multiple collisions \cite{10.1214/EJP.v5-68}. 

The $\Lambda$-coalescent is related to the genealogy of several population models \cite{10.1214/aop/1015345761,schweinsberg2003,DURRETT20051628,Huillet2014,10.1214/22-EJP739} and in particular of $\Lambda$-Fleming-Viot flows \cite{BLGI,labbe2014a,Griffiths2014,vechwkt} and of Continuous State Branching Processes \cite{ref29surveyb,ref36surveyb,ref19surveyb,ref18surveyb}. The process $(\Pi^n_t)_{t\geq 0}$ has been intensely studied, see for example the survey works \cite{Berestycki2009,surveylambdacoal,kerstingwakolbinger2020} and references therein. 

A famous two-parameter family of the $\Lambda$-coalescent is called the Beta-coalescent. It corresponds to choices of $\Lambda$ given by $\beta_{a,b}(\dd r):=r^{a-1}(1-r)^{b-1}\dd r$ for $a,b>0$. The transitions rates from \eqref{coalrates} can then be expressed as 
\begin{align}
\lambda_{p,k}(\beta_{a,b})=B(a+k-2,b+p-k)=\frac{\Gamma(a+k-2)\Gamma(b+p-k)}{\Gamma(a+b+p-2)}, \label{exprelambdapk}
\end{align}
where $B(\cdot,\cdot)$ is the beta function and where we have used the identity $B(x,y)=\Gamma(x)\Gamma(y)/\Gamma(x+y)$. An important case is when $a+b=2$, i.e. $a=2-\alpha$ and $b=\alpha$ for some $\alpha \in (0,2)$, which is called the $Beta(2-\alpha,\alpha)$-coalescent. We note that several authors renormalize the measure by $B(a,b)^{-1}$ in order to make it a probability measure. Multiplication of the measure $\Lambda$ by a constant is harmless (because equivalent to a linear time change of the system), so we decide to not include the renormalization in order to make expressions lighter. 

A quantity of particular interest is the \textit{absorption time}, $\tau_n:=\inf \{ t \geq 0, |\Pi^n_t|=1 \}$ of the $\Lambda$-coalescent started with $n$ blocks, i.e. $\tau_n$ is the time required for the $n$ initial blocks of the $\Lambda$-coalescent $(\Pi^n_t)_{t\geq 0}$ to be gathered into one. From the genealogical point of view, given a population of $n$ individuals whose genealogy can be modeled by a $\Lambda$-coalescent, $\tau_n$ represents how much time one needs to go back in the past to meet the most recent common ancestor of the population. The absorption time $\tau_n$ converges to a finite limit in the case of \textit{coming down from infinity} and to infinity in other cases \cite{schweinsberg2000}. In the case of the Beta-coalescent (i.e. $\Lambda=\beta_{a,b}$), coming down from infinity is equivalent to $a \in (0,1)$. The absorption time of the $\Lambda$-coalescent has been studied by many authors \cite{pitman1999,schweinsberg2000,10.1214/EJP.v10-265,ref27ofsurvey,mohle2014,10.1214/14-AAP1077,kerstingwakolbinger2018}. When the absorption time converges to infinity, it is known that is satisfied a law of large numbers \cite{kerstingwakolbinger2018}. More precisely, the following convergence in probability holds: 
\begin{align}
\frac{\tau_n}{\log n} \overset{\mathbb{P}}{\underset{n \rightarrow \infty}{\longrightarrow}} \frac1{\mu(\Lambda)}, \qquad \text{where} \ \mu(\Lambda):=\int_{(0,1)}|\log(1-r)|r^{-2}\Lambda(\dd r). \label{lgnclassical}
\end{align}
Central limit theorems for the absorption time were also established in \cite{ref27ofsurvey,kerstingwakolbinger2018}. A natural question is to complete these results by a large deviation principle for the absorption time, however, up to our knowledge, no result in that direction has been established so far. Establishing such a large deviation principle in the case of the Beta-coalescent is one of the goals of this paper. Our proof uses completely different ideas from those in previous studies of the absorption time. In particular, our focus is mainly on studying the asymptotic behavior of Laplace transforms of integral functionals of the $\Lambda$-coalescent (of which $\tau_n$ is a particular case). 

In the case of coming down from infinity, convergence of the absorption time in $L^p$ has been studied in \cite{mohle2014} but it seems that the convergence rate has not been studied so far. Our results on the asymptotic behavior of Laplace transforms of integral functionals provide a way to study the convergence in the Kolmogorov distance of $\tau_n$ toward its limit distribution. As a by-product of this approach we also obtain estimates for the record probabilities of the Beta-coalescent that were studied in \cite{10.1214/14-AAP1077}. 

When the $\Lambda$-coalescent $(\Pi^n_t)_{t\geq 0}$ currently contains $k$ blocks, the total transition rate is $\lambda_k(\Lambda):=\sum_{\ell =2}^k \binom{k}{\ell}\lambda_{k,\ell}(\Lambda)$, which can also be expressed as 
\begin{align}
\lambda_k(\Lambda) = \int_{(0,1)} (1-(1-r)^k - kr(1-r)^{k-1}) r^{-2} \Lambda(\dd r). \label{deflambdan}
\end{align}
When, $\Lambda=\beta_{a,b}$, the precise expressions and asymptotics of $\lambda_n(\beta_{a,b})$ are recalled in Appendix \ref{coeflambdanab}. It turns out that asymptotic expansions of quotients of the form $\lambda_n(\beta_{a,b'}))/\lambda_n(\beta_{a,b})$ will play a key role in the proof of our results. 

For any $\Lambda\in\Ms_f([0,1])$, in order to specify the dependency on $\Lambda$, we use the notations $\mathbb{P}_{\Lambda}$ (resp. $\E_{\Lambda}$) for probabilities (resp. expectations) with respect to the $\Lambda$-coalescent. 

\section{Main results}
\subsection{Large deviation principle for the absorption time} \label{sectionldpmain}

Before stating the large deviation result for $\tau_n$ we need to introduce some definitions. Let $L_{a,b}: (-b,\infty) \to \mathbb{R}$ be the function defined by 
\begin{equation} \label{deffcttoinvert}
L_{a,b}(y) := \left\{
\begin{aligned}
& \frac{\Gamma(a)}{(2-a)(a-1)} \left ( \frac{\Gamma(b+y)}{\Gamma(a+b+y-2)} - \frac{\Gamma(b)}{\Gamma(a+b-2)} \right ) \ \text{if} \ a \in (1,\infty) \setminus \{2\}, \\
& \frac{\Gamma'(b+y)}{\Gamma(b+y)} - \frac{\Gamma'(b)}{\Gamma(b)} \ \text{if} \ a=2. 
\end{aligned} \right. 
\end{equation}
Note that the function $\Gamma'(\cdot)/\Gamma(\cdot)$ is known as the \textit{Digamma function}. The Gamma function has a pole at $0$ so, when $a \in (1,2)$ and $y=2-a-b$, the value of $\Gamma(b+y)/\Gamma(a+b+y-2)$ needs to be defined by continuity by setting it to be zero. Similarly, when $a+b=2$ (i.e. case of the $Beta(2-\alpha,\alpha)$-coalescent), the term $\Gamma(b)/\Gamma(a+b-2)$ in \eqref{deffcttoinvert} vanishes. Let us set $D_{a,b}:=\infty$ when $a \in (1,2]$ and $D_{a,b}:=\Gamma(a)\Gamma(b)/\Gamma(a+b-2)(a-2)(a-1)$ when $a>2$. It follows from Lemmas \ref{fctzetaablem}-\ref{convexity}-\ref{fctl2b} (distinguishing cases $a \in (1,2)$, $a>2$, and $a=2$) that $L_{a,b}(\cdot)$ is a $\mathcal{C}^1$ concave-increasing bijection from $(-b,\infty)$ to $(-\infty,D_{a,b})$. We can thus define it's inverse function $\zeta_{a,b}: \mathbb{R} \to (-b,\infty]$ (setting $\zeta_{a,b}(x)=\infty$ when $a>2$ and $x\geq D_{a,b}$). It also follows from Lemmas \ref{fctzetaablem}-\ref{convexity}-\ref{fctl2b} that $\zeta_{a,b}(\cdot)$ is $\mathcal{C}^1$ and convex-increasing on $(-\infty,D_{a,b})$. We note that, for any $k \geq 2$, we have $\lambda_k(\beta_{a,b})<D_{a,b}$ (this is trivial when $a \in (1,2]$ and this follows from Lemma \ref{equlambdnabspecialvalues} when $a>2$). 
We then define a function $\mathcal{I}_{a,b}: \mathbb{R} \to [0,\infty)$ as the Legendre transform of $\zeta_{a,b}(\cdot)$: 
\begin{align}
\mathcal{I}_{a,b}(x)=\sup \{ \theta x - \zeta_{a,b}(\theta); \theta \in \mathbb{R} \}. \label{defiaslegendrezeta}
\end{align}
Note that $\mathcal{I}_{a,b}(x)=\infty$ for $x<0$, $\mathcal{I}_{a,b}(0)=b$, and $\mathcal{I}_{a,b}(x)\in[0,\infty)$ for $x \geq 0$. For any $k \geq 2$, we define a function $\mathcal{I}^k_{a,b}: \mathbb{R} \to [0,\infty]$ by capping the increase rate of $\mathcal{I}_{a,b}(\cdot)$ at $\lambda_k(\beta_{a,b})$. More precisely, it is a classical property of Legendre transform that the derivatives of $\zeta_{a,b}(\cdot)$ and $\mathcal{I}_{a,b}(\cdot)$ are inverse of each others so, setting $x_k:=\zeta'_{a,b}(\lambda_k(\beta_{a,b}))$, we have $\mathcal{I}'_{a,b}(x_k)=\lambda_k(\beta_{a,b})$ and $\mathcal{I}'_{a,b}(x)\geq \lambda_k(\beta_{a,b})$ if and only if $x \geq x_k$. The function $\mathcal{I}^k_{a,b}(\cdot)$ thus has the following expression: 
\begin{equation}\label{defikab}
\mathcal{I}^k_{a,b}(x) := \left\{
\begin{aligned}
& \mathcal{I}_{a,b}(x) \ \text{if} \ x \leq x_k, \\
& \mathcal{I}_{a,b}(x_k)+\lambda_k(\beta_{a,b})(x-x_k) \ \text{if} \ x \geq x_k. \end{aligned} \right. 
\end{equation}
We can now state the large deviation principle satisfied by the absorption times $\tau_n$. 
\begin{theo} \label{thmpgd}
For any $a>1, b>0$, the family $(\tau_n/\log n)_{n \geq 1}$ under $\mathbb{P}_{\beta_{a,b}}$ satisfies a large deviation principle with good rate function $\mathcal{I}^2_{a,b}(\cdot)$ and speed $\log n$. More precisely, for any closed set $F \subset \mathbb{R}$, 
\begin{align}
\limsup_{n \rightarrow \infty} \frac{1}{\log n} \log \mathbb{P}_{\beta_{a,b}} \left ( \frac{\tau_n}{\log n} \in F \right ) \leq -\inf_{x \in F} \mathcal{I}^2_{a,b}(x), \label{ldpupperbound}
\end{align}
and for any open set $G \subset \mathbb{R}$, 
\begin{align}
\liminf_{n \rightarrow \infty} \frac{1}{\log n} \log \mathbb{P}_{\beta_{a,b}} \left ( \frac{\tau_n}{\log n} \in G \right ) \geq -\inf_{x \in G} \mathcal{I}^2_{a,b}(x). \label{ldplowerbound}
\end{align}
\end{theo}
Let us denote by $T_n^k$ the hitting time of $\{1,\ldots,k-1\}$, i.e. $T_n^k:=\inf \{ t \geq 0, |\Pi^n_t|<k \}$. Some aspects of these hitting times were studied in \cite{10.1214/14-AAP1077} where they are called \textit{partial depth of the $n$-coalescent}. Note that $T_n^2$ is just the absorption time $\tau_n$. It is reasonable to wonder what can be said for the hitting time $T_n^k$. This question is answered in the following generalization of Theorem \ref{thmpgd}. 
\begin{theo} \label{thmpgdlevelk}
For any $a>1, b>0$, and $k \geq 2$, the family $(T_n^k/\log n)_{n \geq 1}$ under $\mathbb{P}_{\beta_{a,b}}$ satisfies a large deviation principle with good rate function $\mathcal{I}^k_{a,b}(\cdot)$ and speed $\log n$. More precisely, for any closed set $F \subset \mathbb{R}$, 
\begin{align}
\limsup_{n \rightarrow \infty} \frac{1}{\log n} \log \mathbb{P}_{\beta_{a,b}} \left ( \frac{T_n^k}{\log n} \in F \right ) \leq -\inf_{x \in F} \mathcal{I}^k_{a,b}(x), \label{ldpupperboundlvlk}
\end{align}
and for any open set $G \subset \mathbb{R}$, 
\begin{align}
\liminf_{n \rightarrow \infty} \frac{1}{\log n} \log \mathbb{P}_{\beta_{a,b}} \left ( \frac{T_n^k}{\log n} \in G \right ) \geq -\inf_{x \in G} \mathcal{I}^k_{a,b}(x). \label{ldplowerboundlvlk}
\end{align}
\end{theo}
In the light of Theorem \ref{thmpgdlevelk}, we can offer the following explanation for the cut-off of the rate function at $x_2=\zeta'_{a,b}(\lambda_2(\beta_{a,b}))$ in Theorem \ref{thmpgd}. For $x \leq x_2$, it is only extreme behaviors of $(|\Pi^n_t|)_{t \geq 0}$, while $|\Pi^n_t|$ is still large, that are responsible for the large deviations of the type $\tau_n \approx x \log n$. In contrast, for $x \geq x_2$, excessive values of the holding time of $(|\Pi^n_t|)_{t \geq 0}$ at $2$ also become a source of large deviations of the type $\tau_n \approx x \log n$. 

As one naturally expects, the large deviation principle from Theorem \ref{thmpgd} allows, in the case of the Beta-coalescent, to recover the classical law of large numbers \eqref{lgnclassical}. Indeed, we see from \eqref{defiaslegendrezeta} that $\mathcal{I}_{a,b}(x)\leq 0$ if and only if $\zeta_{a,b}(\theta)\geq \theta x$ for all $\theta \in \mathbb{R}$ and, since $\zeta_{a,b}(0)=0$ and $\zeta_{a,b}(\cdot)$ is convex, this occurs only for $x=\zeta'_{a,b}(0)$. Since $x_k=\zeta'_{a,b}(\lambda_k(\beta_{a,b}))>\zeta'_{a,b}(0)$, we obtain that for any $k \geq 2$, $\mathcal{I}^k_{a,b}(\zeta'_{a,b}(0))=0$ and $\mathcal{I}^k_{a,b}(x)>0$ for all $x \neq \zeta'_{a,b}(0)$. For any $\epsilon>0$, setting $F_{\epsilon}:=(-\infty,\zeta'_{a,b}(0)-\epsilon] \cup [\zeta'_{a,b}(0)+\epsilon,\infty)$ we have $\inf_{x \in F_{\epsilon}} \mathcal{I}^2_{a,b}(x)>0$ so, applying \eqref{ldpupperbound} (or more generally \eqref{ldplowerboundlvlk}) with $F=F_{\epsilon}$ we obtain that, under $\mathbb{P}_{\beta_{a,b}}$, the following convergences in probability holds, 
\begin{align}
\frac{\tau_n}{\log n} \overset{\mathbb{P}}{\underset{n \rightarrow \infty}{\longrightarrow}} \zeta'_{a,b}(0), \qquad \forall k \geq 2, \ \frac{T_n^k}{\log n} \overset{\mathbb{P}}{\underset{n \rightarrow \infty}{\longrightarrow}} \zeta'_{a,b}(0). \label{lgn}
\end{align}
It is justifies in Appendix \ref{appendlgn} that the limit $\zeta'_{a,b}(0)$ from \eqref{lgn} indeed coincides with the limit $1/\mu(\beta_{a,b})$ from the classical result \eqref{lgnclassical}. It is worth noting from \eqref{lgn} and Theorem \ref{thmpgdlevelk} that, interestingly, the hitting times $T_n^k$ (that include $\tau_n$) all satisfy the same law of large numbers but different large deviation principles, as the functions $\mathcal{I}^k_{a,b}(\cdot)$ differ for different values of $k$ but they all have $\zeta'_{a,b}(0)$ as their unique zero. In particular, the large deviation principles from Theorem \ref{thmpgd} contain strictly more information on the behavior of the system than the laws of large numbers \eqref{lgn}. 

\begin{remark} \label{massesofblocks}
It is well-known (see for example \cite[Cor. 3]{pitman1999}) that the $\Lambda$-coalescent $(\Pi^{\infty}_t)_{t\geq 0}$ starting with infinitely many blocks is well-defined as a process on partitions of $\{1,2,\ldots\}$. For each block in $B \in \Pi^{\infty}_t$, its \textit{mass} (also called \textit{asymptotic frequency}) is defined as $|\mathcal{B}|:=\lim_{n \rightarrow \infty}\sharp (\mathcal{B} \cap \{1,\ldots,n\})/n$. It is shown in \cite[Prop. 2.13]{labbe2014a} that blocks masses are well-defined. Studying the masses of blocks at time $t$ when $t$ is large or small, and in particular of the largest block, has attracted some interest \cite{ref19surveyb,vechwkt}. We denote by $W_k(t)$ the mass of the $k^{th}$ largest block at time $t$. It is clear that, as $t$ goes to infinity, one large block occupies a proportion of the mass increasing to $1$ as other massive blocks progressively merge with it, thus $W_1(t) \rightarrow 1$. It is not difficult to see (see for example \cite[Prop 1.6]{vechwkt}) that the absorption time $\tau_n$ is related to the block sizes by $\mathbb{P} (\tau_n \leq t) = \sum_{k \geq 1} \mathbb{E} [W_k(t)^n]$. Therefore, the large deviation principle from Theorem \ref{thmpgd} allows to deduce interesting information on the large deviation principle satisfied by $-\log(1-W_1(t))/t$. Since this goes beyond the scope of the present paper we do not explore this aspect further here. 
\end{remark}

Since Theorem \ref{thmpgd} is included in Theorem \ref{thmpgdlevelk} we only prove the latter (in Section \ref{proofldp}). The most important step in the proof of Theorem \ref{thmpgdlevelk} is to understand the behavior of the Laplace transforms of integral functionals of the $\Lambda$-coalescent. Up to our knowledge, these objects have not been studied in details so far. Our results on these objects are presented in the following section. Once the behavior of these Laplace transforms is established, we can apply the famous G\"artner-Ellis Theorem to obtain the upper bound \eqref{ldpupperboundlvlk} and a partial version of the lower bound \eqref{ldplowerboundlvlk}. Then, a technical difficulty arises that is related to the cut-off of the rate function $\mathcal{I}^k_{a,b}(\cdot)$, but we can complete "manually" the proof of the lower bound \eqref{ldplowerboundlvlk}. 

\subsection{Integral functionals of the Beta-coalescent} \label{sectionintfunct}

For any $\psi:\{2,3,\ldots\} \rightarrow \mathbb{R}$ the $\psi$-integral functional of the $\Lambda$-coalescent is defined by $\int_0^{\tau_n} \psi(|\Pi^n_s|) \dd s$. The object of this section is to provide results for the asymptotic behavior of their Laplace transforms: 
\begin{align}
\Es^{\Lambda}_{n}(\theta \psi):=\E_{\Lambda} [ e^{\theta \int_0^{\tau_n} \psi(|\Pi^n_s|) \dd s} ] \in (0, \infty], \ n \geq 1. \label{defexpfunctional}
\end{align}
Note that $\Es^{\Lambda}_{1}(\theta \psi)=1$. 
For any $j \geq 2$, whenever the state $j$ is visited by the process $(|\Pi^n_s|)_{s \geq 0}$, then the holding time at $j$ follows an exponential distribution with parameter $\lambda_j(\Lambda)$. In particular, if $\theta \psi(j) \geq \lambda_j(\Lambda)$ for some $j \geq 2$ we have $\Es^{\Lambda}_{n}(\theta \psi)=\infty$ for all $n\geq j$. To ensure that $\Es^{\Lambda}_{n}(\theta \psi)<\infty$ for all $n$ we will often assume that the couple $(\psi,\theta)$ satisfies 
\begin{align}
\forall j \geq 2, \ \theta \psi(j) < \lambda_j(\Lambda). \label{condlaplacefinitenosign}
\end{align}
In the case where $\psi(\cdot)$ only takes non-negative values, we often assume that 
\begin{align}
\theta < \inf \{ \lambda_j(\Lambda)/\psi(j); j \geq 2, \psi(j)>0 \} =: M^{\Lambda}(\psi). \label{condlaplacefinite}
\end{align}
For $\psi(\cdot)$ non-negative, \eqref{condlaplacefinite} clearly implies \eqref{condlaplacefinitenosign} and both assumptions are satisfied when $\theta<0$. It will be specified explicitly when our choices of $\psi(\cdot)$ are non-negative. 
\begin{remark}
If $\psi(k)=1$ for all $k \geq 2$ then $\tau_n=\int_0^{\tau_n} \psi(|\Pi^n_s|) \dd s$ so $\theta \mapsto \Es^{\Lambda}_{n}(\theta \psi)$ is the Laplace transform of the absorption time $\tau_n$. This Laplace transform already appeared in \cite[Eq. (32)-(33)]{pitman1999}. In this case, since the sequence $(\lambda_k(\Lambda))_{k \geq 2}$ is increasing (see \eqref{deflambdan}), we have $M^{\Lambda}(\psi)=\lambda_2(\Lambda)$. 
\end{remark}

The following result studies, in the case where $\Lambda=\beta_{a,b}$ for some $a,b>0$, the polynomial growth and decay of $\Es^{\Lambda}_{n}(\theta \psi)$ as $n \to \infty$. 
\begin{theo} \label{asymptfunctional}
Let us fix $\psi:\{2,3,\ldots\} \rightarrow [0,\infty)$ (i.e. $\psi(\cdot)$ is non-negative). 
\begin{itemize}
\item Let $a \in (0,1)$, $b>0$. If there is a constant $c>0$ such that 
\begin{align}
\psi(n) \underset{n \rightarrow \infty}{\sim} c \ n^{1-a}, \label{condpsi0}
\end{align}
then, for any $\theta \in (-b\Gamma(a)/c(2-a)(1-a),M^{\beta_{a,b}}(\psi))$,  
\begin{align}
\frac{\log \Es^{\beta_{a,b}}_{n}(\theta \psi)}{\log n} = \frac{\log \E_{\beta_{a,b}} [ e^{\theta \int_0^{\tau_n} \psi(|\Pi^n_s|) \dd s} ]}{\log n} \underset{n \rightarrow \infty}{\longrightarrow} \frac{\theta c(2-a)(1-a)}{\Gamma(a)}. \label{asymptfunctional0}
\end{align}
\item Let $a=1$, $b>0$. If there is a constant $c>0$ such that 
\begin{align}
\psi(n) \underset{n \rightarrow \infty}{\sim} c \log n, \label{condpsi01}
\end{align}
then, for any $\theta \in (-b/c,M^{\beta_{a,b}}(\psi))$,  
\begin{align}
\frac{\log \Es^{\beta_{a,b}}_{n}(\theta \psi)}{\log n} = \frac{\log \E_{\beta_{a,b}} [ e^{\theta \int_0^{\tau_n} \psi(|\Pi^n_s|) \dd s} ]}{\log n} \underset{n \rightarrow \infty}{\longrightarrow} \theta c. \label{asymptfunctional01}
\end{align}
\item Let $a>1$, $b>0$. If there is a constant $c>0$ such that 
\begin{align}
\psi(n) \underset{n \rightarrow \infty}{\longrightarrow} c, \label{condpsi02}
\end{align}
then, for any $\theta \in (-\infty,M^{\beta_{a,b}}(\psi))$,  
\begin{align}
\frac{\log \Es^{\beta_{a,b}}_{n}(\theta \psi)}{\log n} = \frac{\log \E_{\beta_{a,b}} [ e^{\theta \int_0^{\tau_n} \psi(|\Pi^n_s|) \dd s} ]}{\log n} \underset{n \rightarrow \infty}{\longrightarrow} \zeta_{a,b}(\theta c), \label{asymptfunctional02}
\end{align}
where $\zeta_{a,b}(\cdot)$ is defined in Section \ref{sectionldpmain}. 
\end{itemize}
\end{theo}
We note that, even though $\zeta_{a,b}(\theta c)$ is well-defined for $\theta \in \mathbb{R}$, the asymptotic \eqref{asymptfunctional02} requires $\theta$ to be smaller than $M^{\beta_{a,b}}(\psi)$. Indeed, if $\theta>M^{\beta_{a,b}}(\psi)$, then it follows from the discussion before \eqref{condlaplacefinitenosign} that $\Es^{\beta_{a,b}}_{n}(\theta \psi)=\infty$ for all large $n$. This threshold is thus not an imperfection of our results but it is actually deeply related to the large deviations of the integral functionals $\int_0^{\tau_n} \psi(|\Pi^n_s|) \dd s$ and it results in the cut-off of the rate functions observed in Theorems \ref{thmpgd} and \ref{thmpgdlevelk}. 


The estimates from Theorem \ref{asymptfunctional} provides equivalents only for $\log \Es^{\beta_{a,b}}_{n}(\psi)$. When $a \in (0,1)$ it will be useful to strengthen \eqref{asymptfunctional0} into direct inequalities, provided that $\psi$ satisfies a slightly stronger assumption. This is the object of the following result. Before stating it, let us notice that, the asymptotic of $\lambda_n(\beta_{a,b})$ being given by Lemma \ref{equivlambdanab}, the condition \eqref{condpsi0} can be reformulated as "$\psi(n)/\lambda_n(\beta_{a,b}) \sim d/n$ as $n \to \infty$" 
where $c$ and $d$ are related by $c=d \ \Gamma(a)/(2-a)$. 
\begin{theo} \label{asymptfunctionalbis}
Let $a \in (0,1)$, $b>0$ and $\psi:\{2,3,\ldots\} \rightarrow [0,\infty)$ (i.e. $\psi(\cdot)$ is non-negative). If there is a constant $d \in (0,b/(1-a))$ such that 
\begin{align}
\sum_{n \geq 2} \left | \frac{\psi(n)}{\lambda_n(\beta_{a,b})} - \frac{d}{n} \right | < \infty,  \label{condpsiprecise0}
\end{align}
then there are constants $K_{+},K_{-}\in(0,\infty)$ such that for all $n \geq 1$, 
\begin{align}
\frac{K_{-}}{n^{d(1-a)}} \leq \Es^{\beta_{a,b}}_{n}(-\psi) \leq \frac{K_{+}}{n^{d(1-a)}}. \label{asymptfunctionalprecise0}
\end{align}
\end{theo}

As mentioned in Section \ref{sectionldpmain}, Theorem \ref{asymptfunctional} is a key step in the proof of Theorems \ref{thmpgd}-\ref{thmpgdlevelk}. Theorems \ref{asymptfunctional}-\ref{asymptfunctionalbis} are proved in Section \ref{studyintfct}. We believe that it would not be easy to prove them by a direct study the asymptotic behavior of $\Es^{\Lambda}_{n}(\theta \psi)$ for a given choice of $\psi$. Instead, our strategy is, in Section \ref{changeofmeasuretrick}, to determine target sequences $u=(u_n)_{n \geq 1}$ with known asymptotic behaviors and to identify them as $(\Es^{\Lambda}_{n}(\theta \psi))_{n \geq 1}$ for careful choices of $\psi$. Then, a key step is to determine the asymptotic behavior of these $\psi$ associated to the target sequences. For this, we introduce a trick that is inspired from statistical mechanics (see Remark \ref{statmechtrick} below). We thus obtain a collection of particular cases of $\psi$ that are non-explicit but for which both the asymptotic behaviors of $\psi(n)$ and $\Es^{\Lambda}_{n}(\theta \psi)$ are explicit. Fortunately, this collection of particular cases covers all the behaviors considered in Theorems \ref{asymptfunctional}-\ref{asymptfunctionalbis}, so we can conclude using some monotonicity and stability properties of $\psi \mapsto (\Es^{\Lambda}_{n}(\theta \psi))_{n \geq 1}$ that are established in Section \ref{usefulprop}. 

\begin{remark}[An idea from statistical mechanics] \label{statmechtrick}
In statistical mechanics (see for example \cite{friedlivelenik2017}) a system of size $n$ is considered under a probability measure called \textit{Gibbs distribution} that involves a renormalization constant $Z_n$ called \textit{partition function}. The latter quantity is of great importance, since the Laplace transforms of many relevant quantities can be expressed as a quotient $\tilde Z_n/Z_n$, where $\tilde Z_n$ is the partition function of the system under a change of parameters. Therefore understanding the asymptotic behavior of the partition functions $Z_n$ brings considerable information on the behavior of the system when $n$ is large. In our case, thinking of the quantity $\lambda_n(\beta_{a,b})$ as similar to a partition function, it turns out that, somehow miraculously, we can choose the target sequence $\Es^{\beta_{a,b}}_{n}(\theta \psi)$ so that the corresponding $\psi(n)$ can be expressed in terms of a quotient $\lambda_n(\Lambda_{a,b'})/\lambda_n(\beta_{a,b})$ for some $b'$. Then, as in statistical mechanics, the problem of understanding the behavior of $\psi(n)$ is reduced to determining precisely the asymptotic behavior of quantities $\lambda_n(\beta_{a,b})$, which is very easy to do (see Appendix \ref{coeflambdanab}). 
\end{remark}

The estimates from this section are not a mere lemma for the proof of Theorems \ref{thmpgd}-\ref{thmpgdlevelk}, but they actually have other interesting applications. Indeed, if we enlarge the probability space by letting some event occur at rate $\psi(k)$ when the $\Lambda$-coalescent process $\Pi^n$ has currently $k$ blocks, then $\Es^{\Lambda}_{n}(-\psi)$ represents the probability that the event has not occurred before the absorption time $\tau_n$. Therefore, Theorems \ref{asymptfunctional}-\ref{asymptfunctionalbis} provide us with information on the decay rate of such probabilities. We use this, when $a \in (0,1)$, to study the Kolmogorov distance between $\tau_n$ and its limit distribution. Our results in that direction are presented in the following section. 

\subsection{Convergence rate for the absorption time} \label{cvkolcasecdi}

This section is dedicated to a second application of the results from Section \ref{sectionintfunct}. In the case where the $\beta_{a,b}$-coalescent comes down from infinity, that is, when $a \in (0,1)$, the following result provides a bound for the convergence in the Kolmogorov distance of $\tau_n$ toward its limit distribution. Let us recall that, for two distributions $\mu,\nu$ on $[0,\infty)$ with distribution functions $F_{\mu}, F_{\nu}$, the \textit{Kolmogorov distance} between their distributions is given by 
\begin{align}
d_K(\mu,\nu) := \sup_{t \in [0,\infty)} |F_{\mu}(t)-F_{\nu}(t)|. \label{defkolmogorovdist}
\end{align}
Let us denote by $\rho_n(\Lambda)$ the distribution of $\tau_n$ under $\mathbb{P}_{\Lambda}$ and by $\rho_{\infty}(\Lambda)$ its limit distribution. The latter exists when the $\Lambda$-coalescent comes down from infinity which, for $\Lambda=\beta_{a,b}$, is equivalent to $a \in (0,1)$. The main result of this subsection can be stated as follows. 
\begin{theo} \label{boundkolmdistmain}
Let $a \in (0,1)$ and $b>1-a$. Then there is a constant $C(a,b)\in(0,\infty)$ such that for any $n \geq 1$, 
\begin{align}
d_K(\rho_n(\beta_{a,b}),\rho_{\infty}(\beta_{a,b})) \leq \frac{C(a,b)}{n^{1-a}}. \label{boundkolmdisteq}
\end{align}
\end{theo}
The technical assumption $b>1-a$ is a little restrictive. Fortunately, that assumption is satisfied in the very studied case $a+b=2$, i.e. case of the $Beta(2-\alpha,\alpha)$-coalescent. 

A key step in the proof of Theorem \ref{boundkolmdistmain} is to relate the Kolmogorov distance between $\rho_n(\Lambda)$ and $\rho_{n+1}(\Lambda)$ to a functional of the form \eqref{defexpfunctional}. Before stating a result in this direction, we introduce a notation. For any $\Lambda\in\Ms_f([0,1])$ we denote by $\Lambda'$ the measure in $\Ms_f([0,1])$ defined by $\Lambda'(\dd r)=(1-r)\Lambda(\dd r)$. We note that, for any $a,b>0$, $\beta_{a,b}'=\beta_{a,b+1}$. 

\begin{prop} \label{coaltimentonplus1}
For $\Lambda\in\Ms_f([0,1])$ such that $\Lambda(\{0\})=0$, $n \geq 1$, and $t>0$, we have 
\begin{align} \label{coaltimentonplus1expr}
d_K(\rho_n(\Lambda), \rho_{n+1}(\Lambda)) \leq \Es^{\Lambda'}_{n}(-\phi^{\Lambda}), 
\end{align}
where, for any $k \geq 1$, 
\begin{align} \label{defcoefphilamdak}
\phi^{\Lambda}(k):=\int_{(0,1)}(1-(1-r)^k)r^{-1}\Lambda(\dd r). 
\end{align}
\end{prop}
In view of Proposition \ref{coaltimentonplus1}, the proof of Theorem \ref{boundkolmdistmain} can then be reduced to understanding the asymptotic behavior of the sequence $\phi^{\beta_{a,b}}(n)$ as $n$ goes to infinity and to applying Theorem \ref{asymptfunctionalbis}. We note that Theorem \ref{asymptfunctional} can very well be applied as well but will result in a less precise bound. The proofs of Proposition \ref{coaltimentonplus1} and Theorem \ref{boundkolmdistmain} are given in Section \ref{proofboundkoldist}. 

Finally, let us mention that, for the $\Lambda$-coalescent $(\Pi^{\infty}_t)_{t\geq 0}$ starting with infinitely many blocks (see Remark \ref{massesofblocks}), there is a natural coupling between all the absorption times $(\tau_n)_{n \geq 1}$ that consists in setting $\tau_n$ as the absorption time of the system restricted to the integers $\{1,\ldots,n\}$. Then, the probability $\mathbb{P}_{\Lambda}(\tau_n \neq \tau_{n+1})$ is called \textit{record probability}, as it is the probability that adding level $n+1$ breaks the previous record for the absorption time. Some aspects of the record probabilities are studied in \cite{10.1214/14-AAP1077} in the case of the $\beta_{a,b}$-coalescent where $a+b=2$, i.e. case of the $Beta(2-\alpha,\alpha)$-coalescent. In particular, when $\alpha=3/2$ (i.e. $a=1/2$ and $b=3/2$) the expression of the record probability is derived explicitly: 
\begin{align} \label{recordproba3/2}
\mathbb{P}_{\beta_{1/2,3/2}}(\tau_n \neq \tau_{n+1}) = \frac{3}{2} \frac1{(2n+1)(2n-1)} + \frac{3}{4} \frac{\Gamma(3/2)\Gamma(n)}{\Gamma(n+3/2)}. 
\end{align}
As a by-product of the proofs of Proposition \ref{coaltimentonplus1} and Theorem \ref{boundkolmdistmain}, we obtain the following corollary on the record probabilities. 
\begin{cor} \label{recordproba}
For any $\Lambda\in\Ms_f([0,1])$, and $n \geq 1$, the record probability $\mathbb{P}_{\Lambda}(\tau_n \neq \tau_{n+1})$ equals $\Es^{\Lambda'}_{n}(-\phi^{\Lambda})$ (the latter quantity being as in proposition \ref{coaltimentonplus1}). Moreover, for $\Lambda=\beta_{a,b}$ with $a \in (0,1), b>1-a$, there are constants $c_1(a,b),c_2(a,b)>0$ such that 
\begin{align} \label{recordproba0}
\frac{c_1(a,b)}{n^{2-a}} \leq \mathbb{P}_{\beta_{a,b}}(\tau_n \neq \tau_{n+1}) \leq \frac{c_2(a,b)}{n^{2-a}}. 
\end{align}
\end{cor}
We see from Lemma \ref{equivquotientgamma} that, as $n \to \infty$, the asymptotic of the expression \eqref{recordproba3/2} indeed agrees with Corollary \ref{recordproba}. 

\subsection{Organization of the paper}
The rest of the paper is organized as follows. In Section \ref{studyintfct} we study the Laplace transforms \eqref{defexpfunctional} of the integral functionals and prove the results sated in Section \ref{sectionintfunct}. In Section \ref{proofldp} we apply these results to prove the large deviation principles sated in Section \ref{sectionldpmain}. In Section \ref{proofboundkoldist} we study the convergence rate of the distribution of the absorption time in the case of coming down from infinity and prove the results sated in Section \ref{cvkolcasecdi}. Appendix \ref{coeflambdanab} gathers some expressions and asymptotics of the coefficients $\lambda_n(\beta_{a,b})$. Appendix \ref{fctzetaab} gathers properties of the quotients of two Gamma functions and of the Digamma function. Appendix \ref{appendlgn} identifies the limits from \eqref{lgn} and \eqref{lgnclassical}. 

\section{Integral functionals of the Beta-coalescent} \label{studyintfct}
\subsection{Basic properties} \label{usefulprop}
In this section we establish some basic properties of the Laplace transforms $\Es^{\Lambda}_{n}(\theta \psi)$. In spite of their simplicity, these properties, when combined with the construction of particular cases from Section \ref{changeofmeasuretrick}, will play a role in the derivation of the asymptotics from Theorem \ref{asymptfunctional}. 

For any $k \geq 2$, let $A^n_k$ denote the event where the state $k$ is visited by $(|\Pi^n_t|)_{t \geq 0}$, i.e. 
\begin{align}
A^n_k := \{ \exists t \geq 0 \ \text{s.t.} \ |\Pi^n_t|=k \}. \label{defak}
\end{align}
In order to manage the thresholds in the large deviation principles from Theorems \ref{thmpgd}-\ref{thmpgdlevelk}, we do not only need to study the Laplace transforms of the functionals $\int_0^{\tau_n} \psi(|\Pi^n_s|) \dd s$ under $\mathbb{P}_{\beta_{a,b}}(\cdot)$, but also under $\mathbb{P}_{\beta_{a,b}}(\cdot | A^n_k)$. We thus define, for $\psi:\{2,3,\ldots\} \rightarrow \mathbb{R}$, 
\begin{align}
\Es^{\Lambda}_{n,k}(\theta \psi):=\E_{\Lambda} [ e^{\theta \int_0^{\tau_n} \psi(|\Pi^n_s|) \dd s} \mathds{1}_{A^n_k} ] \in (0, \infty], \ n \geq 1. \label{defexpfunctionalonak}
\end{align}
Assuming \eqref{condlaplacefinitenosign} holds true, we have $\Es^{\Lambda}_{n,k}(\theta \psi)=0$ for all $n<k$, $\Es^{\Lambda}_{k,k}(\theta \psi)=(1-\theta\psi(k)/\lambda_k(\Lambda) )^{-1}$, and $\Es^{\Lambda}_{n,k}(\theta \psi) \in (0, \infty)$ for all $n\geq k$. 

The following lemma provides the recursion satisfied by $\Es^{\Lambda}_{n}(\theta \psi)$ and $\Es^{\Lambda}_{n,k}(\theta \psi)$. The version of that recursion for the case $\psi(k)=1$ already appeared in \cite[Eq. (33)]{pitman1999}. 
\begin{lemma}[recursion] \label{relrecfunctional}
For $\psi:\{2,3,\ldots\} \rightarrow \mathbb{R}$ and $\theta \in \mathbb{R}$ such that \eqref{condlaplacefinitenosign} is satisfied we have, for $n \geq 2$, 
\begin{align}
\Es^{\Lambda}_{n}(\theta \psi) & = \left (1-\theta\frac{\psi(n)}{\lambda_n(\Lambda)} \right )^{-1} \times \sum_{j=1}^{n-1} \binom{n}{j-1} \frac{\lambda_{n,n-j+1}(\Lambda)}{\lambda_n(\Lambda)} \Es^{\Lambda}_{j}(\theta \psi). \label{relrecfunctional0} 
\end{align}
For $\psi:\{2,3,\ldots\} \rightarrow \mathbb{R}$ and $\theta \in \mathbb{R}$ such that \eqref{condlaplacefinitenosign} is satisfied we have, for $n \geq k+1$, 
\begin{align}
\Es^{\Lambda}_{n,k}(\theta \psi) & = \left (1-\theta\frac{\psi(n)}{\lambda_n(\Lambda)} \right )^{-1} \times \sum_{j=k}^{n-1} \binom{n}{j-1} \frac{\lambda_{n,n-j+1}(\Lambda)}{\lambda_n(\Lambda)} \Es^{\Lambda}_{j,k}(\theta \psi). \label{relrecfunctional0onak}
\end{align}
\end{lemma}

\begin{proof}
The claims follow immediately from the Markov property for the block counting process $(|\Pi^{n}_t|)_{t\geq 0}$ at its first transition time, noting that the later follows an exponential distribution with parameter $\lambda_n(\Lambda)$ and that, conditionally on this transition time, transition to the state $j \in \{1,\ldots,n-1\}$ occurs with probability $\lambda_{n,n-j+1}(\Lambda)/\lambda_n(\Lambda)$. 
\end{proof}

The following lemma is a useful comparison principle for $\Es^{\Lambda}_{n}(\cdot)$ and $\Es^{\Lambda}_{n,k}(\cdot)$. 
\begin{lemma}[monotonicity] \label{comparisonprinciple}
For $\psi_1, \psi_2:\{2,3,\ldots\} \rightarrow \mathbb{R}$ and $\theta_1, \theta_2 \in \mathbb{R}$ such that $(\psi_1,\theta_1)$ and $(\psi_2,\theta_2)$ both satisfy \eqref{condlaplacefinitenosign}, if $\theta_1 \psi_1(n) \leq \theta_2 \psi_2(n)$ for all large $n$ then there is a constant $C\in(0,\infty)$ (resp. $C_k\in(0,\infty)$) such that for all $n \geq 1$, $\Es^{\Lambda}_{n}(\theta_1 \psi_1) \leq C \Es^{\Lambda}_{n}(\theta_2 \psi_2)$ (resp. $\Es^{\Lambda}_{n,k}(\theta_1 \psi_1) \leq C_k \Es^{\Lambda}_{n,k}(\theta_2 \psi_2)$). If $\theta_1 \psi_1(n) \leq \theta_2 \psi_2(n)$ for all $n \geq 2$ then we can choose $C=1$ (resp. $C_k=1$). 
\end{lemma}

\begin{proof}
Let us prove the claims for $\Es^{\Lambda}_{n}(\cdot)$. By assumption there is $n_0$ such that $\theta_1 \psi_1(n) \leq \theta_2 \psi_2(n)$ for all $n \geq n_0$. Setting $C:= \max \{ \Es^{\Lambda}_{j}(\theta_1 \psi_1)/\Es^{\Lambda}_{j}(\theta_2 \psi_2); j \in \{1,\ldots,n_0\}\}$ we have $\Es^{\Lambda}_{n}(\theta_1 \psi_1) \leq C \Es^{\Lambda}_{n}(\theta_2 \psi_2)$ for all $n \leq n_0$. Then the first claim follows by induction on $n \geq n_0$ using Lemma \ref{relrecfunctional}. The last claim is a direct consequence of \eqref{defexpfunctional}. The claims for $\Es^{\Lambda}_{n,k}(\cdot)$ are proved similarly setting $C_k:= \max \{ \Es^{\Lambda}_{j,k}(\theta_1 \psi_1)/\Es^{\Lambda}_{j,k}(\theta_2 \psi_2); j \in \{k,\ldots,n_0\}\}$. 
\end{proof}

The following proposition shows that $\Es^{\Lambda}_{n}(-\psi)$ is stable with respect to small variations of $\psi$. 
\begin{lemma}[stability property] \label{stabilityprinciple}
For $\psi_1, \psi_2:\{2,3,\ldots\} \rightarrow [0,\infty)$, if 
\begin{align}
\sum_{n \geq 2} \frac{|\psi_1(n)-\psi_2(n)|}{\lambda_n(\Lambda)} < \infty, \label{condsummability}
\end{align}
then there is a constant $C\geq 1$ such that, for any $n \geq 1$, we have 
\begin{align}
C \ \Es^{\Lambda}_{n}(-\psi_2) \geq \Es^{\Lambda}_{n}(-\psi_1) \geq C^{-1} \ \Es^{\Lambda}_{n}(-\psi_2). \label{equivasymptfunctional}
\end{align}
\end{lemma}

\begin{proof}
We first consider the case of $\psi_1, \psi_2$ satisfying \eqref{condsummability} and the extra assumption 
\begin{align}
\forall n \geq 2, \ \psi_1(n) \geq \psi_2(n). \label{extracond}
\end{align}
We claim that, in this case, for all $n \geq 1$, we have 
\begin{align}
\Es^{\Lambda}_{n}(-\psi_1) \geq \Es^{\Lambda}_{n}(-\psi_2) \times \prod_{k=2}^{n} \left (1+\frac{\psi_1(k)-\psi_2(k)}{\lambda_k(\Lambda)} \right )^{-1}, \label{inductioncompfunctional}
\end{align}
with the convention $\prod_{k=2}^{1} \dots =1$. Once the claim \eqref{inductioncompfunctional} is established, the assumption \eqref{condsummability} ensures that the product in the right-hand side converges to a positive number so, combining with the obvious inequality $\Es^{\Lambda}_{n}(-\psi_2) \geq \Es^{\Lambda}_{n}(-\psi_1)$ from Lemma \ref{comparisonprinciple}, we obtain that, for all $n \geq 1$, 
\begin{align}
\Es^{\Lambda}_{n}(-\psi_2) \geq \Es^{\Lambda}_{n}(-\psi_1) \geq \Es^{\Lambda}_{n}(-\psi_2) \times \prod_{k=2}^{\infty} \left (1+\frac{\psi_1(k)-\psi_2(k)}{\lambda_k(\Lambda)} \right )^{-1}. \label{ineqexactconst}
\end{align}

We now prove the claim \eqref{inductioncompfunctional} by induction. Since $\Es^{\Lambda}_{1}(-\psi_1) = \Es^{\Lambda}_{1}(-\psi_2) =1$, \eqref{inductioncompfunctional} holds for $n=1$. We now assume that \eqref{inductioncompfunctional} holds true for all indices $1,\ldots,n-1$ and justify it for index $n$. We note that 
\begin{align*}
\left (1+\frac{\psi_1(n)}{\lambda_n(\Lambda)} \right )^{-1} \geq \left (1+\frac{\psi_1(n)-\psi_2(n)}{\lambda_n(\Lambda)} \right )^{-1} \times \left (1+\frac{\psi_2(n)}{\lambda_n(\Lambda)} \right )^{-1}. 
\end{align*}
Plugging this and the induction hypothesis in \eqref{relrecfunctional0} (applied to $\Es^{\Lambda}_{n}(-\psi_1)$) we get 
\begin{align*}
\Es^{\Lambda}_{n}(-\psi_1) & \geq \left (1+\frac{\psi_1(n)-\psi_2(n)}{\lambda_n(\Lambda)} \right )^{-1} \\
& \times \left (1+\frac{\psi_2(n)}{\lambda_n(\Lambda)} \right )^{-1} \times \sum_{j=1}^{n-1} \binom{n}{j-1} \frac{\lambda_{n,n-j+1}(\Lambda)}{\lambda_n(\Lambda)} \Es^{\Lambda}_{j}(-\psi_2) \prod_{k=2}^{j} \left (1+\frac{\psi_1(k)-\psi_2(k)}{\lambda_k(\Lambda)} \right )^{-1} \\
& \geq \left ( \prod_{k=2}^{n} \left (1+\frac{\psi_1(k)-\psi_2(k)}{\lambda_k(\Lambda)} \right )^{-1} \right ) \times \left (1+\frac{\psi_2(n)}{\lambda_n(\Lambda)} \right )^{-1} \times \sum_{j=1}^{n-1} \binom{n}{j-1} \frac{\lambda_{n,n-j+1}(\Lambda)}{\lambda_n(\Lambda)} \Es^{\Lambda}_{j}(-\psi_2) \\
& = \Es^{\Lambda}_{n}(-\psi_2) \times \prod_{k=2}^{n} \left (1+\frac{\psi_1(k)-\psi_2(k)}{\lambda_k(\Lambda)} \right )^{-1}, 
\end{align*}
where we have used \eqref{relrecfunctional0} (applied to $\Es^{\Lambda}_{n}(-\psi_2)$) for the last equality. This proves that \eqref{inductioncompfunctional} holds for the index $n$, completing the induction. 

We now consider the case of $\psi_1, \psi_2$ satisfying only the summability assumption \eqref{condsummability}. We define $\psi_0$ via $\psi_0(n) := \psi_2(n) + |\psi_1(n)-\psi_2(n)|$. Then we clearly have that the couples of sequences $(\psi_0, \psi_2)$ and $(\psi_0, \psi_1)$ both satisfy \eqref{condsummability} and \eqref{extracond} so, applying \eqref{ineqexactconst} to these couples, we get that, for all $n \geq 1$, 
\begin{align}
\Es^{\Lambda}_{n}(-\psi_1) & \geq \Es^{\Lambda}_{n}(-\psi_0) \geq \Es^{\Lambda}_{n}(-\psi_2) \times \prod_{k=2}^{\infty} \left (1+\frac{|\psi_1(k)-\psi_2(k)|}{\lambda_k(\Lambda)} \right )^{-1} =: C^{-1} \ \Es^{\Lambda}_{n}(-\psi_2), \label{equivasymptfunctionalhalf}
\end{align}
and this yields one of the bounds from \eqref{equivasymptfunctional}. Then, by reversing the roles of $\psi_1$ and $\psi_2$ in \eqref{equivasymptfunctionalhalf} we obtain the other bound from \eqref{equivasymptfunctional}. 
\end{proof}


\subsection{Construction of a particular case} \label{changeofmeasuretrick}

In view of Lemmas \ref{comparisonprinciple}-\ref{stabilityprinciple}, a path to the proof of Theorems \ref{asymptfunctional}-\ref{asymptfunctionalbis} opens: in order to prove Theorem \ref{asymptfunctional} (resp. Theorem \ref{asymptfunctionalbis}) we only need to find, for $c>0$ (resp. $d>0$), a sequence $\psi(\cdot)$ satisfying \eqref{condpsi0} (resp. \eqref{condpsiprecise0}) and for which the asymptotic behavior of $\Es^{\Lambda}_{n}(\psi)$ can be determined. Then Lemma \ref{comparisonprinciple} (resp. \ref{stabilityprinciple}) ensures that the result can be extended to all sequences satisfying \eqref{condpsi0} (resp. \eqref{condpsiprecise0}) for this $c>0$ (resp. $d>0$). Unfortunately, we are not aware of trivial particular cases of $\psi$ of this kind for all $\Lambda$, and producing one when $\Lambda=\beta_{a,b}$ is already rather involved. This section is dedicated to the construction of such particular cases when $\Lambda=\beta_{a,b}$, and most of the work consists in proving that they indeed satisfy \eqref{condpsi0}-\eqref{condpsiprecise0}, using the trick mentioned in Remark \ref{statmechtrick}. In the following lemma, we choose an explicit sequence that we want to identify as $(\Es^{\beta_{a,b}}_{n}(\psi))_{n\geq 1}$, and build the corresponding $\psi$. 
\begin{lemma} \label{targetsequence}
For any $a,b>0$, $\ell \in (-b,\infty)$, and $n \geq 2$ we set 
\begin{align}
sgn(\ell) \psi_{a,b,\ell}(n) := \lambda_n(\beta_{a,b}) - \sum_{j=1}^{n-1} \binom{n}{j-1} \lambda_{n,n-j+1}(\beta_{a,b}) \frac{\Gamma(b+n-1)\Gamma(b+\ell+j-1)}{\Gamma(b+\ell+n-1)\Gamma(b+j-1)}, \label{defpsiabln}
\end{align}
where $sgn(\ell):=1$ if $\ell \geq 0$ and $-1$ if $\ell<0$. Then $\psi_{a,b,\ell}(\cdot)$ is non-negative, $(\psi_{a,b,\ell},sgn(\ell))$ satisfies \eqref{condlaplacefinitenosign}, and for any $n \geq 1$ we have 
\begin{align}
\Es^{\beta_{a,b}}_{n}(sgn(\ell)\psi_{a,b,\ell}) = \frac{\Gamma(b+\ell+n-1)\Gamma(b)}{\Gamma(b+n-1)\Gamma(b+\ell)}, \label{targetsequence0}
\end{align}
so in particular 
\begin{align}
\Es^{\beta_{a,b}}_{n}(sgn(\ell)\psi_{a,b,\ell}) \underset{n \rightarrow \infty}{\sim} \frac{\Gamma(b)}{\Gamma(b+\ell)} n^{\ell}. \label{targetsequencefinal}
\end{align}
\end{lemma}

\begin{proof}
Using the identity $\lambda_n(\Lambda)=\sum_{\ell =2}^n \binom{n}{\ell}\lambda_{n,\ell}(\Lambda)$ and Lemma \ref{fctzetaablem} we see that the quantity $\psi_{a,b,\ell}(n)$ defined in \eqref{defpsiabln} is always positive when $\ell\neq0$ and null when $\ell=0$. The positivity of the second term in the right-hand side of \eqref{defpsiabln} shows that $(\psi_{a,b,\ell},sgn(\ell))$ satisfies \eqref{condlaplacefinitenosign}. Since $\Es^{\beta_{a,b}}_{1}(sgn(\ell)\psi_{a,b,\ell})=1$, we see that \eqref{targetsequence0} holds true for $n=1$. By induction, for $k \geq 2$, if \eqref{targetsequence0} holds true for all $n \in \{1,\ldots,k-1\}$ then we can re-write \eqref{defpsiabln}, at $n=k$, as
\begin{align*}
sgn(\ell)\psi_{a,b,\ell}(k) = \lambda_k(\beta_{a,b}) - \frac{\Gamma(b+k-1)\Gamma(b+\ell)}{\Gamma(b+\ell+k-1)\Gamma(b)} \sum_{j=1}^{k-1} \binom{k}{j-1} \lambda_{k,k-j+1}(\beta_{a,b}) \Es^{\beta_{a,b}}_{j}(\psi_{a,b,\ell}). 
\end{align*}
Comparing this with Lemma \ref{relrecfunctional} (applied at $n=k$) we obtain that \eqref{targetsequence0} holds for $n=k$, concluding the proof by induction. Then, \eqref{targetsequencefinal} follows from \eqref{targetsequence0} and Lemma \ref{equivquotientgamma}. 
\end{proof}

The following lemma provides the asymptotic behavior of the sequence $\psi_{a,b,\ell}$ from Lemma \ref{targetsequence}. 
\begin{lemma} \label{behavspecialpsi}
Let $\psi_{a,b,\ell}$ be as in \eqref{defpsiabln} and $L_{a,b}(\cdot)$ as in \eqref{deffcttoinvert}. 
\begin{itemize}
\item For any $a \in (0,2) \setminus \{1\}$, $b>0$ and $\ell \in (-b,\infty)$, we have 
\begin{align}
sgn(\ell) \psi_{a,b,\ell}(n) & = \frac{\Gamma(a) \ell}{(2-a)(1-a)} n^{1-a} + L_{a,b}(\ell) + \underset{n \rightarrow \infty}{o}\left (n^{0\wedge (1-a)}\right ). \label{behavspecialpsi3}
\end{align}
\item For $a=1$, $b>0$, $\ell \in (-b,\infty)$, we have $sgn(\ell) \psi_{a,b,\ell}(n) \sim \ell \log n$ as $n \rightarrow \infty$. 
\item For $a\geq 2$, $b>0$, $\ell \in (-b,\infty)$, we have $sgn(\ell) \psi_{a,b,\ell}(n) \to L_{a,b}(\ell)$ as $n \rightarrow \infty$. 
\end{itemize}
\end{lemma}

\begin{proof}
Using \eqref{exprelambdapk} we see that 
\begin{align}
\lambda_{n,n-j+1}(\beta_{a,b}) \frac{\Gamma(b+\ell+j-1)}{\Gamma(b+j-1)} & = \frac{\Gamma(a+n-j-1)\Gamma(b+j-1) \Gamma(b+\ell+j-1)}{\Gamma(a+b+n-2) \Gamma(b+j-1)} \nonumber \\ 
& = \frac{\Gamma(a+b+\ell+n-2)}{\Gamma(a+b+n-2)} \frac{\Gamma(a+n-j-1) \Gamma(b+\ell+j-1)}{\Gamma(a+b+\ell+n-2)} \nonumber \\ 
& = \frac{\Gamma(a+b+\ell+n-2)}{\Gamma(a+b+n-2)} \lambda_{n,n-j+1}(\beta_{a,b+\ell}). \label{exprlambdanj}
\end{align}
Using \eqref{defpsiabln} and combining with \eqref{exprlambdanj} and the identity $\lambda_n(\Lambda)=\sum_{\ell =2}^n \binom{n}{\ell}\lambda_{n,\ell}(\Lambda)$, we get 
\begin{align}
1-sgn(\ell) \frac{\psi_{a,b,\ell}(n)}{\lambda_n(\beta_{a,b})} & = \frac{\Gamma(b+n-1)}{\Gamma(b+\ell+n-1)} \frac1{\lambda_n(\beta_{a,b})} \sum_{j=1}^{n-1} \binom{n}{j-1} \lambda_{n,n-j+1}(\beta_{a,b}) \frac{\Gamma(b+\ell+j-1)}{\Gamma(b+j-1)} \nonumber \\
& = \frac{\Gamma(b+n-1)}{\Gamma(b+\ell+n-1)} \frac{\Gamma(a+b+\ell+n-2)}{\Gamma(a+b+n-2)} \frac1{\lambda_n(\beta_{a,b})} \sum_{j=1}^{n-1} \binom{n}{j-1} \lambda_{n,n-j+1}(\beta_{a,b+\ell}) \nonumber \\
& = \frac{\Gamma(b+n-1)}{\Gamma(b+\ell+n-1)} \frac{\Gamma(a+b+\ell+n-2)}{\Gamma(a+b+n-2)} \frac{\lambda_{n}(\beta_{a,b+\ell})}{\lambda_n(\beta_{a,b})}. \label{exprquotient}
\end{align}
Using Lemma \ref{equivquotientgamma} and that $a>0$, we get 
\begin{align}
\frac{\Gamma(b+n-1)}{\Gamma(b+\ell+n-1)} & = n^{-\ell} \left ( 1 - \frac{\ell (2b-3+\ell)}{2n} + \underset{n \rightarrow \infty}{o}\left (\frac1{n^{1\vee(2-a)}}\right ) \right ), \label{asympquotient1} \\
\frac{\Gamma(a+b+\ell+n-2)}{\Gamma(a+b+n-2)} & = n^{\ell} \left ( 1 + \frac{\ell (2a+2b+\ell-5)}{2n} + \underset{n \rightarrow \infty}{o}\left (\frac1{n^{1\vee(2-a)}}\right ) \right ). \label{asympquotient2}
\end{align}
We now assume that $a \in (0,2) \setminus \{1\}$. We get from Lemma \ref{equivlambdanab} that
\begin{align}
\frac{\lambda_{n}(\beta_{a,b+\ell})}{\lambda_n(\beta_{a,b})} = 1 - \frac{(2-a)a \ell}{n(1-a)} + \frac{J(a,b+\ell)-J(a,b)}{n^{2-a}} + \underset{n \rightarrow \infty}{o}\left (\frac1{n^{1\vee(2-a)}}\right ). \label{asympquotient3}
\end{align}
Plugging \eqref{asympquotient1}, \eqref{asympquotient2} and \eqref{asympquotient3} into \eqref{exprquotient} we get, after a few cancellations, 
\begin{align}
1-sgn(\ell) \frac{\psi_{a,b,\ell}(n)}{\lambda_n(\beta_{a,b})} = 1 - \frac{\ell}{n(1-a)} + \frac{J(a,b+\ell)-J(a,b)}{n^{2-a}} + \underset{n \rightarrow \infty}{o}\left (\frac1{n^{1\vee(2-a)}}\right ). \label{asympquotient0}
\end{align}
Simplifying the $1$ on both sides of \eqref{asympquotient0}, multiplying by $-\lambda_n(\beta_{a,b})$, using Lemma \ref{equivlambdanab} again, and using the definitions of $J(\cdot,\cdot)$ and $L_{a,b}(\cdot)$ (in respectively Lemma \ref{equivlambdanab} and \eqref{deffcttoinvert}) we obtain \eqref{behavspecialpsi3}. 

We now assume that $a=1$ and get from Lemma \ref{equlambdnabspecialvalues} that
\begin{align}
\frac{\lambda_{n}(\beta_{a,b+\ell})}{\lambda_n(\beta_{a,b})} = 1 - \ell \frac{\log n}{n} +  \underset{n \rightarrow \infty}{o} \left (\frac{\log n}{n}\right ). \label{asympquotient3aequal1}
\end{align}
Plugging \eqref{asympquotient1}, \eqref{asympquotient2} and \eqref{asympquotient3aequal1} into \eqref{exprquotient} we get, 
\begin{align}
1-sgn(\ell) \frac{\psi_{a,b,\ell}(n)}{\lambda_n(\beta_{a,b})} = 1 - \ell \frac{\log n}{n} +  \underset{n \rightarrow \infty}{o} \left (\frac{\log n}{n}\right ). \label{asympquotient0aequal1}
\end{align}
Simplifying the $1$ on both sides of \eqref{asympquotient0aequal1}, multiplying by $-\lambda_n(\beta_{a,b})$, and using Lemma \ref{equlambdnabspecialvalues} again we obtain $sgn(\ell) \psi_{a,b,\ell}(n) \sim \ell \log n$ as claimed. 

We now assume that $a=2$ and get from Lemma \ref{equlambdnabspecialvalues} that
\begin{align}
\frac{\lambda_{n}(\beta_{a,b+\ell})}{\lambda_n(\beta_{a,b})} = 1 + \frac{H(b+\ell)-H(b)}{\log n} + \underset{n \rightarrow \infty}{o}\left (\frac1{\log n}\right ). \label{asympquotient3aequal2}
\end{align}
Plugging \eqref{asympquotient1}, \eqref{asympquotient2} and \eqref{asympquotient3aequal2} into \eqref{exprquotient} we get, 
\begin{align}
1-sgn(\ell) \frac{\psi_{a,b,\ell}(n)}{\lambda_n(\beta_{a,b})} = 1 + \frac{H(b+\ell)-H(b)}{\log n} + \underset{n \rightarrow \infty}{o}\left (\frac1{\log n}\right ). \label{asympquotient0aequal2}
\end{align}
Then we can conclude as in the other cases and get $sgn(\ell) \psi_{a,b,\ell}(n) \to L_{2,b}(\ell)$ as claimed. 

We now assume that $a>2$ and get from Lemma \ref{equlambdnabspecialvalues} that
\begin{align}
\frac{\lambda_{n}(\beta_{a,b+\ell})}{\lambda_n(\beta_{a,b})} \underset{n \rightarrow \infty}{\longrightarrow} \frac{\Gamma(b+\ell)\Gamma(a+b-2)}{\Gamma(b)\Gamma(a+b+\ell-2)}. \label{asympquotient3a>2}
\end{align}
Plugging \eqref{asympquotient1}, \eqref{asympquotient2} and \eqref{asympquotient3a>2} into \eqref{exprquotient} we get, 
\begin{align}
1-sgn(\ell) \frac{\psi_{a,b,\ell}(n)}{\lambda_n(\beta_{a,b})} \underset{n \rightarrow \infty}{\longrightarrow} \frac{\Gamma(b+\ell)\Gamma(a+b-2)}{\Gamma(b)\Gamma(a+b+\ell-2)}. \label{asympquotient0a>2}
\end{align}
Then we can conclude as in the other cases and get $sgn(\ell) \psi_{a,b,\ell}(n) \to L_{a,b}(\ell)$ as claimed. 
\end{proof}

The following lemma is analogue to Lemma \ref{targetsequence} but for $\Es^{\Lambda}_{n,k}(\cdot)$. Since the proof is identical to the proof of Lemma \ref{targetsequence} we omit it. 
\begin{lemma} \label{targetsequenceonak}
For any $a,b>0$, $\ell \in (-b,\infty)$, and $k \geq 2$ we set $\psi_{a,b,\ell,k}(n)=0$ when $n\in \{2,\ldots,k\}$ and, for $n>k$, 
\begin{align}
sgn(\ell) \psi_{a,b,\ell,k}(n) := \lambda_n(\beta_{a,b}) - \sum_{j=k}^{n-1} \binom{n}{j-1} \lambda_{n,n-j+1}(\beta_{a,b}) \frac{\Gamma(b+n-1)\Gamma(b+\ell+j-1)}{\Gamma(b+\ell+n-1)\Gamma(b+j-1)}, \label{defpsiablnonak}
\end{align}
where $sgn(\ell):=1$ if $\ell \geq 0$ and $sgn(\ell):=-1$ if $\ell<0$. Then $(\psi_{a,b,\ell,k},sgn(\ell))$ satisfies \eqref{condlaplacefinitenosign}, and, for any $n \geq k$, we have 
\begin{align}
\Es^{\beta_{a,b}}_{n,k}(sgn(\ell)\psi_{a,b,\ell,k}) = \frac{\Gamma(b+\ell+n-1)\Gamma(b+k-1)}{\Gamma(b+n-1)\Gamma(b+\ell+k-1)}, \label{targetsequence0onak}
\end{align}
so in particular 
\begin{align}
\Es^{\beta_{a,b}}_{n,k}(sgn(\ell)\psi_{a,b,\ell,k}) \underset{n \rightarrow \infty}{\sim} \frac{\Gamma(b+k-1)}{\Gamma(b+\ell+k-1)} n^{\ell}. \label{targetsequencefinalonak}
\end{align}
\end{lemma}
We note that the sequence $\psi_{a,b,\ell,k}(\cdot)$ from Lemma \ref{targetsequenceonak} is not guaranteed to be non-negative everywhere but this will not be an issue. The following lemma is analogue to Lemma \ref{behavspecialpsi} but for $\psi_{a,b,\ell,k}$. Since we only need the result in the case $a>1$, the lemma only studies this case. 
\begin{lemma} \label{behavspecialpsionak}
For any $a>1$, $b>0$, $\ell \in (-b,\infty)$ and $k\geq 2$, let $\psi_{a,b,\ell,k}$ be as in \eqref{defpsiablnonak}. Then we have $sgn(\ell) \psi_{a,b,\ell,k}(n) \to L_{a,b}(\ell)$ as $n \to \infty$, 
where $L_{a,b}(\cdot)$ is as in \eqref{deffcttoinvert}. 
\end{lemma}

\begin{proof}
Using \eqref{defpsiablnonak} and proceeding as in \eqref{exprquotient} we get 
\begin{align}
1-sgn(\ell) \frac{\psi_{a,b,\ell,k}(n)}{\lambda_n(\beta_{a,b})} & = \frac{\Gamma(b+n-1)}{\Gamma(b+\ell+n-1)} \frac{\Gamma(a+b+\ell+n-2)}{\Gamma(a+b+n-2)} \nonumber \\
& \times \frac{\lambda_{n}(\beta_{a,b+\ell})-\sum_{j=1}^{k-1} \binom{n}{j-1} \lambda_{n,n-j+1}(\beta_{a,b+\ell})}{\lambda_n(\beta_{a,b})}. \label{exprquotientonak}
\end{align}
By \eqref{exprelambdapk} and Lemma \ref{equivquotientgamma} we have 
\begin{align}
\sum_{j=1}^{k-1} \binom{n}{j-1} \lambda_{n,n-j+1}(\beta_{a,b+\ell}) = \frac{C}{n^{b+\ell}} + \underset{n \rightarrow \infty}{o} \left (\frac1{n^{b+\ell}}\right ), \label{extraterms}
\end{align}
where $C$ is a positive constant whose precise value is unimportant. Combining \eqref{extraterms} with Lemma \ref{equivlambdanab} we get, when $a \in (1,2)$,  
\begin{align}
\frac{\lambda_{n}(\beta_{a,b+\ell})-\sum_{j=1}^{k-1} \binom{n}{j-1} \lambda_{n,n-j+1}(\beta_{a,b+\ell})}{\lambda_n(\beta_{a,b})} = 1 + \frac{J(a,b+\ell)-J(a,b)}{n^{2-a}} + \underset{n \rightarrow \infty}{o}\left (\frac1{n^{2-a}}\right ). \label{asympquotient3onak}
\end{align}
Plugging \eqref{asympquotient1}, \eqref{asympquotient2} and \eqref{asympquotient3onak} into \eqref{exprquotientonak} we get, after a few cancellations, 
\begin{align*}
1-sgn(\ell) \frac{\psi_{a,b,\ell,k}(n)}{\lambda_n(\beta_{a,b})} = 1 + \frac{J(a,b+\ell)-J(a,b)}{n^{2-a}} + \underset{n \rightarrow \infty}{o}\left (\frac1{n^{2-a}}\right ), 
\end{align*}
so, when $a \in (1,2)$, we can conclude as in the proof of Lemma \ref{behavspecialpsi}. Similarly, when $a=2$ (resp. $a>2$), combining \eqref{extraterms} with Lemma \ref{equlambdnabspecialvalues} we get that the estimate \eqref{asympquotient3aequal2} (resp. \eqref{asympquotient3a>2}) still holds when $\lambda_{n}(\beta_{a,b+\ell})/\lambda_n(\beta_{a,b})$ is replaced by the left-hand side of \eqref{asympquotient3onak}. We then deduce that \eqref{asympquotient0aequal2} (resp. \eqref{asympquotient0a>2}) holds true when $\psi_{a,b,\ell}(n)$ is replaced by $\psi_{a,b,\ell,k}(n)$ and we can conclude as in the proof of Lemma \ref{behavspecialpsi}. 
\end{proof}

\subsection{Proof of Theorems \ref{asymptfunctional} and \ref{asymptfunctionalbis}}

We can now put pieces together and prove the main results about the asymptotic behavior of the Laplace transforms of the integral functionals. 

\begin{proof}[Proof of Theorem \ref{asymptfunctionalbis}]
Let $a,b, d$ and $\psi$ be as in the statement of the theorem. In particular $\psi$ satisfies \eqref{condpsiprecise0} with this value of $d$. We set $\ell:=-d(1-a)$ and consider $\psi_{a,b,\ell}$ as defined in \eqref{defpsiabln}. By \eqref{asympquotient0} we get that $\psi_{a,b,\ell}$ also satisfies \eqref{condpsiprecise0} with the current value of $d$. Therefore the couple of sequences $(\psi, \psi_{a,b,\ell})$ satisfies \eqref{condsummability} so, by Lemma \ref{stabilityprinciple}, there is $C>1$ such that for all $n \geq 1$, $C \Es^{\beta_{a,b}}_{n}(-\psi_{a,b,\ell}) \geq \Es^{\beta_{a,b}}_{n}(-\psi) \geq C^{-1} \Es^{\beta_{a,b}}_{n}(-\psi_{a,b,\ell})$. 
The combination of this and \eqref{targetsequencefinal} yields \eqref{asymptfunctionalprecise0}, concluding the proof. 
\end{proof}

\begin{proof}[Proof of Theorem \ref{asymptfunctional}]
\textit{Case $a \in (0,1)$, $b>0$.} Let $\psi:\{2,3,\ldots\} \rightarrow [0,\infty)$ and $c>0$ such that \eqref{condpsi0} holds true. Let also $\theta \in (-b\Gamma(a)/c(2-a)(1-a),M^{\beta_{a,b}}(\psi))$. We set $\ell_1(\epsilon):=-\epsilon+\theta c(2-a)(1-a)/\Gamma(a)$ and $\ell_2(\epsilon):=\epsilon+\theta c(2-a)(1-a)/\Gamma(a)$, where $\epsilon>0$ is chosen small enough so that $\ell_1(\epsilon)>-b$ and $\ell_1(\epsilon),\ell_2(\epsilon)\neq0$. We consider $\psi_{a,b,\ell_1(\epsilon)}$ and $\psi_{a,b,\ell_2(\epsilon)}$ as defined in \eqref{defpsiabln}, but with $\ell$ replaced by respectively $\ell_1(\epsilon)$ and $\ell_2(\epsilon)$. By Lemma \ref{behavspecialpsi} we get that $\psi_{a,b,\ell_1(\epsilon)}$ and $\psi_{a,b,\ell_2(\epsilon)}$ satisfy 
\begin{align*}
& sgn(\ell_1(\epsilon)) \psi_{a,b,\ell_1(\epsilon)}(n) \underset{n \rightarrow \infty}{\sim} \left ( \theta c - \frac{\Gamma(a) \epsilon}{(2-a)(1-a)} \right ) n^{1-a}, \\ 
& sgn(\ell_2(\epsilon)) \psi_{a,b,\ell_2(\epsilon)}(n) \underset{n \rightarrow \infty}{\sim} \left ( \theta c + \frac{\Gamma(a) \epsilon}{(2-a)(1-a)} \right ) n^{1-a}. 
\end{align*}
Therefore we have $sgn(\ell_1(\epsilon)) \psi_{a,b,\ell_1(\epsilon)}(n) \leq \theta \psi(n) \leq sgn(\ell_2(\epsilon)) \psi_{a,b,\ell_2(\epsilon)}(n)$ for all large $n$. We deduce from Lemma \ref{comparisonprinciple} that there are constants $c_1, c_2>0$ such that, for all $n \geq 1$, 
\begin{align}
c_1 \Es^{\beta_{a,b}}_{n}(sgn(\ell_1(\epsilon)) \psi_{a,b,\ell_1(\epsilon)}) \leq \Es^{\beta_{a,b}}_{n}(\theta \psi) \leq c_2 \Es^{\beta_{a,b}}_{n}(sgn(\ell_2(\epsilon)) \psi_{a,b,\ell_2(\epsilon)}). \label{compineqfinalthetapos}
\end{align}
Recall from Lemma \ref{targetsequence} that $\Es^{\beta_{a,b}}_{n}(sgn(\ell_1(\epsilon)) \psi_{a,b,\ell_1(\epsilon)})$ and $\Es^{\beta_{a,b}}_{n}(sgn(\ell_2(\epsilon)) \psi_{a,b,\ell_2})$ satisfy \eqref{targetsequencefinal} (but with $\ell$ replaced by respectively $\ell_1(\epsilon)$ and $\ell_2(\epsilon)$). Now taking the logarithm in \eqref{compineqfinalthetapos}, diving by $\log n$, and using the \eqref{targetsequencefinal} we get 
\begin{align*}
\ell_1(\epsilon) \leq \liminf_{n \rightarrow \infty} \frac{\log \Es^{\beta_{a,b}}_{n}(\theta \psi)}{\log n} \leq \limsup_{n \rightarrow \infty} \frac{\log \Es^{\beta_{a,b}}_{n}(\theta \psi)}{\log n} \leq \ell_2(\epsilon). 
\end{align*}
Then letting $\epsilon$ go to $0$ yields \eqref{asymptfunctional0}. This concludes the proof in the case $a \in (0,1)$, $b>0$. 

\textit{Case $a=1$, $b>0$.} This case is proved exactly as the previous case, using that in this case the estimate $sgn(\ell) \psi_{a,b,\ell}(n) \sim \ell \log n$ from Lemma \ref{behavspecialpsi} holds instead of \eqref{behavspecialpsi3}. 

\textit{Case $a>1$, $b>0$.} Let $\psi:\{2,3,\ldots\} \rightarrow [0,\infty)$ and $c>0$ such that \eqref{condpsi02} holds true. Let also $\theta \in (-\infty,M^{\beta_{a,b}}(\psi))$. We set $\ell_1(\epsilon):=-\epsilon+\zeta_{a,b}(\theta c)$ and $\ell_2(\epsilon):=\epsilon+\zeta_{a,b}(\theta c)$, where $\epsilon>0$ is chosen small enough so that $\ell_1(\epsilon)>-b$ and $\ell_1(\epsilon),\ell_2(\epsilon)\neq0$. As before we consider $\psi_{a,b,\ell_1(\epsilon)}$ and $\psi_{a,b,\ell_2(\epsilon)}$ and get from Lemma \ref{behavspecialpsi} that $\psi_{a,b,\ell_1(\epsilon)}$ and $\psi_{a,b,\ell_2(\epsilon)}$ satisfy 
\begin{align*}
& sgn(\theta) \psi_{a,b,\ell_1(\epsilon)}(n) \underset{n \rightarrow \infty}{\longrightarrow} L_{a,b}(-\epsilon+\zeta_{a,b}(\theta c)), \\ 
& sgn(\theta) \psi_{a,b,\ell_2(\epsilon)}(n) \underset{n \rightarrow \infty}{\longrightarrow} L_{a,b}(\epsilon+\zeta_{a,b}(\theta c)). 
\end{align*}
Also, as $n \to \infty$, $\theta \psi(n) \to \theta c=L_{a,b}(\zeta_{a,b}(\theta c))$ by \eqref{condpsi02} and definition of $\zeta_{a,b}(\cdot)$. Since $L_{a,b}(\cdot)$ is an increasing function by Lemma \ref{fctzetaablem}, we get $sgn(\ell_1(\epsilon)) \psi_{a,b,\ell_1(\epsilon)}(n) \leq \theta \psi(n) \leq sgn(\ell_2(\epsilon)) \psi_{a,b,\ell_2(\epsilon)}(n)$ for all large $n$. Then we can concludes the proof as in the case $a \in (0,1)$, $b>0$ and get \eqref{asymptfunctional02}. 
\end{proof}


Recall the event $A^n_k$ defined in \eqref{defak}. When $a \in (1,2)$, the following conditional version of Theorem \ref{asymptfunctional} will be useful for the proof of Theorems \ref{thmpgd} and \ref{thmpgdlevelk}. 
\begin{prop} \label{asymptfunctionalonak}
Let us fix $\psi:\{2,3,\ldots\} \rightarrow [0,\infty)$, $a>1$, $b>0$. If there is a constant $c>0$ such that \eqref{condpsi02} is satisfied, then, for any $\theta \in (-\infty,M^{\beta_{a,b}}(\psi))$,  
\begin{align}
\frac{\log \E_{\beta_{a,b}} [ e^{\theta \int_0^{\tau_n} \psi(|\Pi^n_s|) \dd s}|A^n_k ]}{\log n} \underset{n \rightarrow \infty}{\longrightarrow} \zeta_{a,b}(\theta c). \label{asymptfunctional02onak}
\end{align}
\end{prop}
\begin{proof}
Proceeding as in the proof of Theorem \ref{asymptfunctional} (case $a>1$, $b>0$) but using $\psi_{a,b,\ell_i(\epsilon),k}$ and Lemmas \ref{targetsequenceonak}-\ref{behavspecialpsionak} instead of $\psi_{a,b,\ell_i(\epsilon)}$ and Lemmas \ref{targetsequence}-\ref{behavspecialpsi}, we get that the estimate \eqref{asymptfunctional02} is also satisfied for $\Es^{\beta_{a,b}}_{n,k}(\theta \psi)$ instead of $\Es^{\beta_{a,b}}_{n}(\theta \psi)$. Since $\E_{\beta_{a,b}} [ e^{\theta \int_0^{\tau_n} \psi(|\Pi^n_s|) \dd s}|A^n_k ]=\Es^{\beta_{a,b}}_{n,k}(\theta \psi)/\Es^{\beta_{a,b}}_{n,k}(0)$, and $\zeta_{a,b}(0)=0$, the result follows. 
\end{proof}

\begin{remark} \label{asymptpak}
For $n,k \geq 2$, $\mathbb{P}_{\Lambda} ( A^n_k )$ is the probability that $k$ belongs to the range of the $\Lambda$-coalescent started with $n$ blocks. As a by-product of the previous proof we get that, for $a>1$, $b>0$, and $k \geq 2$, 
\begin{align}
\frac{\log \mathbb{P}_{\beta_{a,b}} ( A^n_k )}{\log n} \underset{n \rightarrow \infty}{\longrightarrow} 0. \label{asymptpak0}
\end{align}
This will come useful in the proof of Theorems \ref{thmpgd}-\ref{thmpgdlevelk}. It is proved in \cite[Cor. 3.6]{10.1214/14-AAP1077} that, in the case where $a+b=2$ (i.e. case of the $Beta(2-\alpha,\alpha)$-coalescent), the probability $\mathbb{P}_{\beta_{a,b}} ( A^n_k )$ converges as $n \to \infty$ and an expression of the limit is given. It is reasonable to bet that such a convergence also holds in the case $a>1$, $b>0$. However, we will not need such a precise estimate in the proof of Theorems \ref{thmpgd}-\ref{thmpgdlevelk}, and \eqref{asymptpak0} will be sufficient for our purpose. 
\end{remark}

\section{Large deviation principle for the absorption time} \label{proofldp}

In this section we prove Theorem \ref{thmpgdlevelk} (which contains Theorem \ref{thmpgd}). For any $k \geq 2$ we set $\psi_k(n):=\mathds{1}_{[k,\infty)}(n)$ and we note that, for any $n \geq 1$, $T^n_k=\int_0^{\tau_n} \psi_k(|\Pi^n_s|) \dd s$. The first step of the proof consists in applying Theorem \ref{asymptfunctional} to this integral functional and in combining that result with the G\"artner-Ellis Theorem. This will yield the upper bound \eqref{ldpupperboundlvlk} and a partial version of the lower bound \eqref{ldplowerboundlvlk}. The second step of the proof consists into completing the lower bound. 

\begin{proof}[Proof of Theorem \ref{thmpgdlevelk}]
We fix $a>1$, $b>0$, and $k \geq 2$. For $\psi_k$ as in the previous discussion we have $M^{\beta_{a,b}}(\psi_k)=\lambda_k(\beta_{a,b})$. By the discussion before \eqref{condlaplacefinitenosign}, if $\theta \geq \lambda_k(\beta_{a,b})$, we have $\Es^{\beta_{a,b}}_{n}(\theta \psi_k)=\infty$ for all $n\geq k$. Also, we trivially have that $\psi_k$ satisfies \eqref{condpsi02} with $c=1$ so the estimate \eqref{asymptfunctional02} from Theorem \ref{asymptfunctional} holds true. Combining all this we get that, for any $\theta \in \mathbb{R}$,  
\begin{align}
\frac{\log \E_{\beta_{a,b}} [ e^{\theta T^n_k} ]}{\log n} \underset{n \rightarrow \infty}{\longrightarrow} \zeta_{a,b}(\theta) + \infty \mathds{1}_{\theta\geq \lambda_k(\beta_{a,b})}, \label{asymptfunctional02psik}
\end{align}
where $\zeta_{a,b}(\cdot)$ is defined in Section \ref{sectionldpmain}, and where we use the convention $\infty \times 0=0$. The function in the right-hand side of \eqref{asymptfunctional02psik} is finite on the domain $(-\infty,\lambda_k(\beta_{a,b}))$ which contains the origin. Therefore it satisfies Assumption 2.3.2 of \cite{dembozeitouni}. Moreover, it is not difficult to see that the Legendre transform of that function is precisely the function $\mathcal{I}^k_{a,b}(\cdot)$ defined in \eqref{defikab}. Therefore, the upper bound in G\"artner-Ellis Theorem (see \cite[Thm. 2.3.6(a)]{dembozeitouni}) yields directly the upper bound \eqref{ldpupperboundlvlk}. 

The lower bound in G\"artner-Ellis Theorem (see \cite[Thm. 2.3.6(b)]{dembozeitouni}) yields that, for any open set $G \subset \mathbb{R}$, 
\begin{align}
\liminf_{n \rightarrow \infty} \frac{1}{\log n} \log \mathbb{P}_{\beta_{a,b}} \left ( \frac{T_n^k}{\log n} \in G \right ) \geq -\inf_{x \in G \cap \mathcal{F}_k} \mathcal{I}^k_{a,b}(x), \label{ldplowerboundlvlkpartial}
\end{align}
where $\mathcal{F}_k$ is, in the terminology of \cite[Def. 2.3.3]{dembozeitouni}, the set of \textit{exposed points} of $\mathcal{I}^k_{a,b}(\cdot)$ whose \textit{exposing hyperplane} belongs to $(-\infty,\lambda_k(\beta_{a,b}))$. 
We see from the discussion before \eqref{defiaslegendrezeta} that $\zeta'_{a,b}(\cdot)$ is continuous on $(-\infty,\lambda_k(\beta_{a,b}))$. Moreover $\zeta'_{a,b}(\theta) \to 0$ as $\theta \to -\infty$ and $\zeta'_{a,b}(\lambda_k(\beta_{a,b}))=x_k$ by the definition of $x_k$ just before \eqref{defikab}. Therefore, for any $x \in (0,x_k)$ there is a $\theta \in (-\infty,\lambda_k(\beta_{a,b}))$ such that $\zeta'_{a,b}(\theta)=x$. According to \cite[Lem. 2.3.9]{dembozeitouni} we deduce that $(0,x_k) \subset \mathcal{F}_k$. 

To prove \eqref{ldplowerboundlvlk} we only need to show that, for any open set $G \subset \mathbb{R}$ and $\epsilon>0$ we have 
\begin{align}
\liminf_{n \rightarrow \infty} \frac{1}{\log n} \log \mathbb{P}_{\beta_{a,b}} \left ( \frac{T_n^k}{\log n} \in G \right ) \geq -\inf_{x \in G} \mathcal{I}^k_{a,b}(x)-\epsilon. \label{ldplowerboundlvlkepsilon}
\end{align}
We thus fix an open set $G \subset \mathbb{R}$ and $\epsilon>0$. If $G \subset (-\infty,0)$ there is nothing to prove so we assume that $G \cap [0,\infty) \neq \emptyset$. If there is $z \in (0,x_k) \cap G$ such that 
\begin{align}
\mathcal{I}^k_{a,b}(z) < \inf_{x \in G} \mathcal{I}^k_{a,b}(x)+\epsilon, \label{approxinf}
\end{align}
then \eqref{ldplowerboundlvlkepsilon} follows from \eqref{ldplowerboundlvlkpartial} and from $(0,x_k) \subset \mathcal{F}_k$. If there is no such $z$ in $(0,x_k)$, then there is $z \in (x_k,\infty) \cap G$ satisfying \eqref{approxinf}. Let us fix $\delta \in (0,z\wedge x_{k+1}-x_k)$ such that $(z-\delta,z+\delta)\subset G$. Recall the event $A^n_k$ defined in \eqref{defak}. By the Markov property at the hitting time of $k$ by $(|\Pi^n_t|)_{t \geq 0}$ we see that, conditionally on $A^n_k$, $T^n_k$ is distributed as $e_k+F_k$ where $e_k$ and $F_k$ are independent, $e_k$ follows an exponential distribution with parameter $\lambda_k(\beta_{a,b})$, and $F_k$ is distributed as $\int_0^{\tau_n} \psi_{k+1}(|\Pi^n_s|) \dd s$ conditionally on $A^n_k$. We thus get 
\begin{align}
\mathbb{P}_{\beta_{a,b}} \left ( \frac{T_n^k}{\log n} \in G \right ) & \geq \mathbb{P}_{\beta_{a,b}} \left ( A^n_k \right ) \times \mathbb{P}_{\beta_{a,b}} \left ( \frac{T_n^k}{\log n} \in (z-\delta,z+\delta) | A^n_k \right ) \nonumber \\
& \geq \mathbb{P}_{\beta_{a,b}} \left ( A^n_k \right ) \times \mathbb{P}_{\beta_{a,b}} \left ( \frac{\int_0^{\tau_n} \psi_{k+1}(|\Pi^n_s|) \dd s}{\log n} \in (x_k-\delta/2,x_k+\delta/2) | A^n_k \right ) \nonumber \\ 
& \qquad \qquad \qquad \times \mathbb{P}_{\beta_{a,b}} \left ( \frac{e_k}{\log n} \in (z-x_k-\delta/2,z-x_k+\delta/2) \right ). \label{manualldplowerbound}
\end{align}
Then we have 
\begin{align*}
\mathbb{P}_{\beta_{a,b}} \left ( \frac{e_k}{\log n} \in (z-x_k-\delta/2,z-x_k+\delta/2) \right ) = e^{-\lambda_k(\beta_{a,b})(z-x_k-\delta/2) \log n} - e^{-\lambda_k(\beta_{a,b})(z-x_k+\delta/2) \log n}, 
\end{align*}
so 
\begin{align}
\liminf_{n \rightarrow \infty} \frac{1}{\log n} \log \mathbb{P}_{\beta_{a,b}} \left ( \frac{e_k}{\log n} \in (z-x_k-\delta/2,z-x_k+\delta/2) \right ) \geq -\lambda_k(\beta_{a,b})(z-x_k). \label{asymptexpo}
\end{align}

Proceeding as in \eqref{asymptfunctional02psik}, but for $\psi_{k+1}$ and $\mathbb{P}_{\beta_{a,b}}(\cdot | A^n_k)$ instead of $\psi_k$ and $\mathbb{P}_{\beta_{a,b}}(\cdot)$ (and, in particular, using Proposition \ref{asymptfunctionalonak} instead of Theorem \ref{asymptfunctional}) we get that for any $\theta \in \mathbb{R}$,  
\begin{align}
\frac{\log \E_{\beta_{a,b}} [ e^{\theta \int_0^{\tau_n} \psi_{k+1}(|\Pi^n_s|) \dd s}  | A^n_k]}{\log n} \underset{n \rightarrow \infty}{\longrightarrow} \zeta_{a,b}(\theta) + \infty \mathds{1}_{\theta\geq \lambda_{k+1}(\beta_{a,b})}. \label{asymptfunctional02psikcond}
\end{align}
We can thus apply again G\"artner-Ellis Theorem. The lower bound part (see \cite[Thm. 2.3.6(b)]{dembozeitouni}) yields that, for any open set $\tilde G \subset \mathbb{R}$, 
\begin{align*}
\liminf_{n \rightarrow \infty} \frac{1}{\log n} \log \mathbb{P}_{\beta_{a,b}} \left ( \frac{\int_0^{\tau_n} \psi_{k+1}(|\Pi^n_s|) \dd s}{\log n} \in \tilde G | A^n_k \right ) \geq -\inf_{x \in \tilde G \cap \mathcal{F}_{k+1}} \mathcal{I}^{k+1}_{a,b}(x). 
\end{align*}
Since $x_k \in (x_k-\delta/2,x_k+\delta/2) \subset (0,x_{k+1}) \subset \mathcal{F}_{k+1}$ we obtain 
\begin{align}
& \liminf_{n \rightarrow \infty} \frac{1}{\log n} \log \mathbb{P}_{\beta_{a,b}} \left ( \frac{\int_0^{\tau_n} \psi_{k+1}(|\Pi^n_s|) \dd s}{\log n} \in (x_k-\delta/2,x_k+\delta/2) | A^n_k \right ) \nonumber \\
& \qquad \qquad \qquad \geq -\inf_{x \in (x_k-\delta/2,x_k+\delta/2)} \mathcal{I}^{k+1}_{a,b}(x) \geq -\mathcal{I}^{k+1}_{a,b}(x_k) = -\mathcal{I}_{a,b}(x_k), \label{ldplowerboundlvlkcond}
\end{align}
where the last equality comes from the definition of $\mathcal{I}^{\cdot}_{a,b}(\cdot)$ in \eqref{defikab}. 

Finally, taking the logarithm in \eqref{manualldplowerbound}, dividing by $\log n$ and using the estimates from \eqref{asymptexpo}, \eqref{ldplowerboundlvlkcond}, and \eqref{asymptpak0} we obtain 
\begin{align*}
\liminf_{n \rightarrow \infty} \frac{1}{\log n} \log \mathbb{P}_{\beta_{a,b}} \left ( \frac{T_n^k}{\log n} \in G \right ) \geq -(\mathcal{I}_{a,b}(x_k) + \lambda_k(\beta_{a,b})(z-x_k)) = -\mathcal{I}^k_{a,b}(z), 
\end{align*}
where the last equality comes from the definition of $\mathcal{I}^k_{a,b}(\cdot)$ in \eqref{defikab}. Since $z$ has been chosen so that it satisfies \eqref{approxinf}, we deduce that \eqref{ldplowerboundlvlkepsilon} holds true, concluding the proof. 

\end{proof}

\section{Convergence rate for the absorption time} \label{proofboundkoldist}
\subsection{Connection with an integral functional: Proof of Proposition \ref{coaltimentonplus1}}

For $\Lambda\in\Ms_f([0,1])$ such that $\Lambda(\{0\})=0$, let us recall the natural construction, initially from \cite[Cor. 3]{pitman1999}, of the $\Lambda$-coalescent started with $n$ blocks. Let $N$ be a Poisson point process on $(0,\infty) \times (0,1) \times (\{0,1\}^{n})$ with intensity measure $m(ds,dr,dz):=ds \times (\mathcal{B}(r)^{\times n}(dz)) r^{-2} \Lambda(dr)$, where we write $\mathcal{B}(r)^{\times n}$ for the distribution of a sequence of $n$ iid Bernoulli random variables with parameter $r$. We build a process $(\Pi^{n}_t(N))_{t\geq 0}$ of random partitions of $\{1,\ldots,n\}$ where, at any time $t$, the blocks of the partitions are denoted by $A^1_t,A^2_t,\ldots$ and ordered by their smallest element. At any time $t$ the number of non-empty blocks $|\Pi^{n}_t(N)|$ is smaller or equal to $n$ (for convenience we set $A^i_t=\emptyset$ for all $i \in \{|\Pi^{n}_t(N)|+1,\ldots,n\}$). Let $Z^n_0:=\{ z=(z_i)_{1 \leq i \leq n} \in \{0,1\}^{n}; \exists i,j \ \text{s.t.} \ i \neq j \ \text{and} \ z_i=z_j=1\}$. Then the set of points of $N$ in $(0,\infty) \times (0,1) \times Z^n_0$ is discrete. For each such point $(s,r,z)$, we select all the blocks $A^i_{s-}$ such that $z_i=1$ and let them merge at time $s$, leaving other blocks untouched. No other transitions of the system are allowed. By \cite[Cor. 3]{pitman1999}, the process $(\Pi^{n}_t(N))_{t\geq 0}$ constructed above is a $\Lambda$-coalescent started with $n$ blocks. 

\begin{proof}[Proof of Proposition \ref{coaltimentonplus1}]
We fix $n \geq 1$. Let $N$ and $(\Pi^{n+1}_t(N))_{t\geq 0}$ be as in the discussion above, with the exception that, if one of the blocks at time $t$ is $\{n+1\}$, we shift the index of that block to $n+1$, so that $A^{n+1}_t=\{n+1\}$. Clearly, $(\Pi^{n+1}_t(N))_{t\geq 0}$ is a $\Lambda$-coalescent started with $n+1$ blocks and its restriction $(\Pi^{n}_t(N))_{t\geq 0}$ to $\{1,\ldots,n\}$ is a $\Lambda$-coalescent started with $n$ blocks. Then for any $t \geq 0$ we have 
\begin{align*}
\mathbb{P}_{\Lambda}(\tau_{n+1}>t) = \mathbb{P}(|\Pi^{n+1}_t(N)|>1) & = \mathbb{P}(|\Pi^{n}_t(N)|>1) + \mathbb{P}(|\Pi^{n}_t(N)|=1, |\Pi^{n+1}_t(N)|>1) \\
& = \mathbb{P}_{\Lambda}(\tau_n>t) + \mathbb{P}(|\Pi^{n}_t(N)|=1, |\Pi^{n+1}_t(N)|=2). 
\end{align*}
In particular, 
\begin{align}
\mathbb{P}_{\Lambda}(\tau_n\leq t)-\mathbb{P}_{\Lambda}(\tau_{n+1}\leq t) = \mathbb{P}(|\Pi^{n}_t(N)|=1, |\Pi^{n+1}_t(N)|=2). \label{decomptaunplus1}
\end{align}
For $i \in \{0,1\}$, let $S_i := (0,\infty) \times (0,1) \times (\{0,1\}^{n} \times \{i\})$ and let $N_i$ denote the restriction of $N$ to $S_i$, projected on $S:=(0,\infty) \times (0,1) \times (\{0,1\}^{n})$. Then $N_0$ and $N_1$ are independent Poisson random measures on $S$ with intensity respectively $m_0(ds,dr,dz):=ds \times (\mathcal{B}(r)^{\times n}(dz)) r^{-2} \Lambda'(dr)$ and $m_1(ds,dr,dz):=ds \times (\mathcal{B}(r)^{\times n}(dz)) r^{-1} \Lambda(dr)$. In particular, $(\Pi^{n}_t(N_0))_{t\geq 0}$ is a $\Lambda'$-coalescent started with $n$ blocks. Any atom of $N_0$ (resp. $N_1$) represents a potential merging event that does not (resp. does) involve the block with index $n+1$. The event $\{|\Pi^{n}_t(N)|=1, |\Pi^{n+1}_t(N)|=2\}$ signifies that the first $n$ blocks have merged within time $t$ while the block $\{n+1\}$ has not merged with any other non-empty block on $[0,t]$. Therefore, this event means that $N_1$ has contributed to no non-trivial merger on $[0,t]$, while $N_0$ has contributed to merging the first $n$ blocks. 
If there are currently $k$ blocks in $\Pi^{n}_t(N)$, then $N_1$ contributes to a non-trivial merger at rate $\phi^{\Lambda}(k)$, where $\phi^{\Lambda}(\cdot)$ is a in \eqref{defcoefphilamdak}. We thus get 
\begin{align}
\mathbb{P}(|\Pi^{n}_t(N)|=1, |\Pi^{n+1}_t(N)|=2) & = \E \left [ \mathds{1}_{|\Pi^{n}_t(N_0)|=1} e^{-\int_0^t \phi^{\Lambda}(|\Pi^n_s(N_0)|) \dd s} \right ] \nonumber \\
& = \E_{\Lambda'} \left [ \mathds{1}_{|\Pi^{n}_t|=1} e^{-\int_0^t \phi^{\Lambda}(|\Pi^n_s|) \dd s} \right ] \leq \E_{\Lambda'} \left [ e^{-\int_0^{\tau_n} \phi^{\Lambda}(|\Pi^n_s|) \dd s} \right ]. \label{lastestimntonplus1}
\end{align}
Then combining \eqref{defkolmogorovdist}, \eqref{decomptaunplus1} and \eqref{lastestimntonplus1} we obtain \eqref{coaltimentonplus1expr}. 
\end{proof}

\subsection{Estimate for the Kolmogorov distance: Proof of Theorem \ref{boundkolmdistmain}}

The next step in the proof of Theorem \ref{boundkolmdistmain} is to determine the asymptotic behaviors of $\phi^{\beta_{a,b}}(n)$ as $n \to \infty$. This is done in the following lemma. 
\begin{lemma} \label{equivphinab}
For any $a \in (0,1)$ and $b>0$ we have 
\begin{align}
\phi^{\beta_{a,b}}(n) & = \frac{\Gamma(a)}{1-a} n^{1-a} \left ( 1 - \mathds{1}_{a+b\neq 1} \frac{\Gamma(b)}{\Gamma(a+b-1)n^{1-a}} + \frac{(1-a)(a+2b-2)}{2n} + \underset{n \rightarrow \infty}{o} \left (\frac1{n}\right ) \right ). \label{equivphinab2}
\end{align}
\end{lemma}

\begin{proof}
For any $a \in (0,1)$, $b>0$ and $n \geq 2$ we have 
\begin{align} \label{calphinab1}
\phi^{\beta_{a,b}}(n) & = \frac{\Gamma(a)}{1-a} \left ( \frac{\Gamma(n+b)}{\Gamma(n-1+a+b)} - \mathds{1}_{a+b\neq 1} \frac{\Gamma(b)}{\Gamma(a+b-1)} \right ). 
\end{align}
Indeed, we have $1-(1-r)^n = r ( \sum_{k=0}^{n-1} (1-r)^k )$ so, combining with \eqref{defcoefphilamdak}, we get 
\begin{align*}
\phi^{\beta_{a,b}}(n) = \sum_{j=0}^{n-1} \int_{(0,1)} (1-r)^j \beta_{a,b}(\dd r) = \sum_{j=1}^{n} B(a,b+j-1) = \Gamma(a) \sum_{j=1}^{n} \frac{\Gamma(b-1+j)}{\Gamma(a+b-1+j)}. 
\end{align*}
Then using \cite[Lem. A.1]{10.1214/24-EJP1198} we get \eqref{calphinab1}. Then \eqref{equivphinab2} follows from the combination of \eqref{calphinab1} with Lemma \ref{equivquotientgamma}. 
\end{proof}

We can now bring the pieces together to prove Theorem \ref{boundkolmdistmain}. 
\begin{proof}[Proof of Theorem \ref{boundkolmdistmain}]
Let $a,b$ be as in the statement of the theorem. Using Proposition \ref{coaltimentonplus1} and the fact that $\beta_{a,b}'=\beta_{a,b+1}$ we get that 
\begin{align}
d_K(\rho_n(\beta_{a,b}),\rho_{\infty}(\beta_{a,b})) \leq \sum_{k \geq n} d_K(\rho_k(\beta_{a,b}),\rho_{k+1}(\beta_{a,b})) \leq \sum_{k \geq n} \Es^{\beta_{a,b+1}}_{k}(-\phi^{\beta_{a,b}}). \label{boundkolmdisteq1}
\end{align}
Lemma \ref{equivphinab} shows that $\phi^{\beta_{a,b}}$ satisfies \eqref{condpsi0} with $c=\Gamma(a)/(1-a)$. 
By the combination of Lemmas \ref{equivphinab} and \ref{equivlambdanab} we get  
\begin{align}
\frac{\phi^{\beta_{a,b}}(n)}{\lambda_n(\beta_{a,b+1})} = \frac{2-a}{1-a} \frac1{n} \left ( 1 - \mathds{1}_{a+b\neq 1} \frac{\Gamma(b)}{\Gamma(a+b-1)n^{1-a}} + \underset{n \rightarrow \infty}{o} \left (\frac1{n^{1-a}}\right ) \right ). \label{boundkolmdisteq2}
\end{align}
Since $a \in (0,1)$, this shows in particular that $\phi^{\beta_{a,b}}(n)/\lambda_n(\beta_{a,b+1})$ satisfies \eqref{condpsiprecise0} with $d=(2-a)/(1-a)$. Since $a+b>1$, we have moreover that $d \in (0,(b+1)/(1-a))$ so, by Theorem \ref{asymptfunctionalbis}, there are $K_{+},K_{-}>0$ such that for all $k \geq 1$, 
\begin{align}
\frac{K_{-}}{k^{2-a}} \leq \Es^{\beta_{a,b+1}}_{k}(-\phi^{\beta_{a,b}}) \leq \frac{K_{+}}{k^{2-a}}. \label{boundkolmdisteq3}
\end{align}
The combination of \eqref{boundkolmdisteq1} and \eqref{boundkolmdisteq3} yields 
\begin{align*}
d_K(\rho_n(\beta_{a,b}),\rho_{\infty}(\beta_{a,b})) \leq K_{+} \sum_{k \geq n} \frac{1}{k^{2-a}} \underset{n \rightarrow \infty}{\sim} \frac{K_{+}}{(1-a) n^{1-a}}, 
\end{align*}
and this yields \eqref{boundkolmdisteq}. 
\end{proof}

\begin{proof}[Proof of Corollary \ref{recordproba}]
The identification $\mathbb{P}_{\Lambda}(\tau_n \neq \tau_{n+1})=\Es^{\Lambda'}_{n}(-\phi^{\Lambda})$ follows from the same argument proving Proposition \ref{coaltimentonplus1}. Then, the bound \eqref{recordproba0} follows from this identification together with \eqref{boundkolmdisteq3}. 
\end{proof}

\appendix
\section{ } \label{append}

\subsection{The coefficients $\lambda_n(\beta_{a,b})$} \label{coeflambdanab}
In this appendix we gather some expressions and asymptotics of the coefficients $\lambda_n(\beta_{a,b})$ (see \eqref{deflambdan}). For the sake of completeness, we provide all the details of the derivations of these formulas and asymptotics. 

\begin{lemma} \label{calclambdanab}
For any $a \in (0,\infty) \setminus \{1,2\}$, $b>0$ and $n \geq 2$ we have 
\begin{align}
\lambda_n(\beta_{a,b}) & = \frac{\Gamma(a)}{2-a} \frac{\Gamma(n+b)}{\Gamma(n-2+a+b)} - \frac{(b-1)\Gamma(a)}{1-a} \frac{\Gamma(n-1+b)}{\Gamma(n-2+a+b)} + \frac{\Gamma(a)\Gamma(b)}{\Gamma(a+b-2)} \frac1{(1-a)(2-a)}, \label{calclambdanab1}
\end{align}
For $a=1$, $b>0$ and $n \geq 2$ we have 
\begin{align}
\lambda_n(\beta_{a,b}) & = n-1 - (b-1) \sum_{j=1}^{n-1} \frac{1}{b-1+j}. \label{calclambdanab1caseaequal1}
\end{align}
For $a=2$, $b>0$ and $n \geq 2$ we have 
\begin{align}
\lambda_n(\beta_{a,b}) & = \sum_{j=1}^{n-1} \frac{1}{b+j} - 1 + \frac{1}{b} + \frac{b-1}{b+n-1}. \label{calclambdanab1caseaequal2}
\end{align}
\end{lemma}
The Gamma function has a poles at $0$ and $-1$ but $1/\Gamma(0)$ and $1/\Gamma(-1)$ can be defined by continuity by setting them to be zero. In particular, the expression in \eqref{calclambdanab1} is indeed well-defined when $a+b \in \{1,2\}$. 
\begin{proof}
We have 
\begin{align*}
1-(1-r)^n-nr(1-r)^{n-1} & = r \left ( \sum_{k=0}^{n-1} (1-r)^k \right ) - nr(1-r)^{n-1} = r \sum_{k=0}^{n-2} \left ( (1-r)^k - (1-r)^{n-1} \right ) \\
& = r^2 \sum_{k=0}^{n-2} \sum_{j=k}^{n-2} (1-r)^j = r^2 \sum_{j=0}^{n-2} (j+1) (1-r)^j. 
\end{align*}
Combining with \eqref{deflambdan} we get, for any $a,b>0$ and $n \geq 2$, 
\begin{align*}
\lambda_n(\beta_{a,b}) = \sum_{j=0}^{n-2} (j+1) \int_{(0,1)} (1-r)^j \beta_{a,b}(\dd r) = \sum_{j=1}^{n-1} j B(a,b+j-1) = \Gamma(a) \sum_{j=1}^{n-1} j \frac{\Gamma(b-1+j)}{\Gamma(a+b-1+j)}, 
\end{align*}
We first consider the case where $a \in (0,\infty) \setminus \{1,2\}$ and $b>0$. From the above we get 
\begin{align}
\lambda_n(\beta_{a,b}) & = \Gamma(a) \left ( \sum_{j=1}^{n-1} (b-1+j) \frac{\Gamma(b-1+j)}{\Gamma(a+b-1+j)} - (b-1) \sum_{j=1}^{n-1} \frac{\Gamma(b-1+j)}{\Gamma(a+b-1+j)} \right ) \nonumber \\
& = \Gamma(a) \left ( \sum_{j=1}^{n-1} \frac{\Gamma(b+j)}{\Gamma(a+b-1+j)} - (b-1) \sum_{j=1}^{n-1} \frac{\Gamma(b-1+j)}{\Gamma(a+b-1+j)} \right ) \label{pointer} \\
& = \frac{\Gamma(a)}{2-a} \left ( \frac{\Gamma(b+n)}{\Gamma(a+b+n-2)} - \mathds{1}_{a+b\neq 1} \frac{\Gamma(b+1)}{\Gamma(a+b-1)} \right ) \nonumber \\
& - \frac{(b-1)\Gamma(a)}{1-a} \left ( \frac{\Gamma(b+n-1)}{\Gamma(a+b+n-2)} - \mathds{1}_{a+b\neq 1} \frac{\Gamma(b)}{\Gamma(a+b-1)} \right ), \nonumber
\end{align}
where we have used $(b-1+j)\Gamma(b-1+j)=\Gamma(b+j)$ for the second equality and \cite[Lem. A.1]{10.1214/24-EJP1198} for the third equality. 
The terms depending on $n$ are identical to the ones announced in \eqref{calclambdanab1}, it thus remains to calculate the constant term. The latter equals 
\begin{align*}
& \mathds{1}_{a+b\neq 1} \frac{\Gamma(a)}{\Gamma(a+b-1)} \left ( - \frac{\Gamma(b+1)}{2-a} + \frac{(b-1)\Gamma(b)}{1-a} \right ) \\
= & \mathds{1}_{a+b\neq 1} \frac{\Gamma(a)\Gamma(b)}{\Gamma(a+b-1)} \left ( - \frac{b}{2-a} + \frac{b-1}{1-a} \right ) \\
= & \mathds{1}_{a+b\neq 1} \frac{\Gamma(a)\Gamma(b)}{\Gamma(a+b-1)} \frac{a+b-2}{(1-a)(2-a)} = \mathds{1}_{a+b \notin \{1,2\}} \frac{\Gamma(a)\Gamma(b)}{\Gamma(a+b-2)} \frac1{(1-a)(2-a)}. 
\end{align*}
This yields \eqref{calclambdanab1}. 

If $a=1$ and $b>0$, note that \eqref{pointer} still holds true. Setting $a=1$ in that expression and using $(b-1+j)\Gamma(b-1+j)=\Gamma(b+j)$ we get \eqref{calclambdanab1caseaequal1}. Similarly, if $a=2$ and $b>0$, setting $a=2$ in \eqref{pointer} and using the same property of the Gamma function we get 
\begin{align*}
\lambda_n(\beta_{a,b}) & = \sum_{j=1}^{n-1} \frac{1}{b+j} - (b-1) \sum_{j=1}^{n-1} \frac{1}{(b-1+j)(b+j)}. 
\end{align*}
Simplifying the second sum yields \eqref{calclambdanab1caseaequal2}. 

\end{proof}

\begin{lemma} \label{equivlambdanab}
For any $a \in (0,2)\setminus\{1\}$, $b>0$ and $n \geq 2$ we have 
\begin{align} \label{calclambdanab2}
\lambda_n(\beta_{a,b}) = \frac{\Gamma(a)}{2-a} n^{2-a} \left ( 1 + \frac{2-a}{2n}\left (a+2b-3-\frac{2(b-1)}{1-a}\right ) + \frac{J(a,b)}{n^{2-a}} + \underset{n \rightarrow \infty}{o}\left (\frac1{n^{1\vee(2-a)}}\right ) \right ), 
\end{align}
where $J(a,b) := \Gamma(b)/\Gamma(a+b-2)(1-a)$ (we note that $1/\Gamma(a+b-2)$ is understood as $0$ when $a+b \in \{1,2\}$). In particular we have that $\lambda_n(\beta_{a,b}) \sim \Gamma(a)(2-a)^{-1} n^{2-a}$ as $n \rightarrow \infty$. 
\end{lemma}

\begin{proof}
This follows from the combination of Lemmas \ref{calclambdanab} and \ref{equivquotientgamma}. 
\end{proof}

\begin{lemma} \label{equlambdnabspecialvalues}
\begin{itemize}
\item For $a=1$ and $b>0$ we have 
\begin{align}
\lambda_n(\beta_{a,b}) = n \left ( 1 -(b-1) \frac{\log n}{n} +  \underset{n \rightarrow \infty}{o} \left (\frac{\log n}{n}\right ) \right ). \label{equivlambdanabaequal11}
\end{align}
In particular we have that $\lambda_n(\beta_{a,b}) \sim n$ as $n \rightarrow \infty$. 
\item For $a=2$ and $b>0$ we have 
\begin{align}
\lambda_n(\beta_{a,b}) = \log n \left ( 1 + \frac{H(b)}{\log n} +  \underset{n \rightarrow \infty}{o} \left (\frac{1}{\log n}\right ) \right ), \label{equivlambdanabaequal12}
\end{align}
where $H(b)=-1-\Gamma'(b)/\Gamma(b)$. In particular we have that $\lambda_n(\beta_{a,b}) \sim \log n$ as $n \rightarrow \infty$. 
\item For $a>2$ and $b>0$ we have 
\begin{align}
\lambda_n(\beta_{a,b}) \underset{n \rightarrow \infty}{\longrightarrow} \frac{\Gamma(a)\Gamma(b)}{\Gamma(a+b-2)} \frac1{(a-1)(a-2)}. \label{equivlambdanaba>2}
\end{align}
Moreover, for any $n \geq 2$, $\lambda_n(\beta_{a,b})$ is strictly smaller than its limit. 
\end{itemize}
\end{lemma}

\begin{proof}
The estimate \eqref{equivlambdanabaequal11} follows directly from Lemma \ref{calclambdanab}. Similarly, the estimate \eqref{equivlambdanaba>2} follows directly from Lemmas \ref{calclambdanab} and \ref{equivquotientgamma} (the claim that $\lambda_n(\beta_{a,b})$ is strictly smaller than its limit comes from $\lambda_n(\beta_{a,b})$ being increasing in $n$ by \eqref{deflambdan}). For \eqref{equivlambdanabaequal12}, using \eqref{calclambdanab1caseaequal2} and the fact that $-\log n + \sum_{j=1}^{n-1} \frac{1}{j}=\gamma + o(1)$ as $n \to \infty$ (where $\gamma$ is Euler's constant) we get 
\begin{align*}
\lambda_n(\beta_{a,b}) & = - 1 + \frac{1}{b} + \log n + \gamma - \sum_{j=1}^{n-1} \left ( \frac{1}{j} - \frac{1}{b+j} \right ) + \underset{n \rightarrow \infty}{o} (1) \\
& = - 1 + \frac{1}{b} + \log n + \gamma - \sum_{j=1}^{\infty} \frac{b}{j(b+j)} + \underset{n \rightarrow \infty}{o} (1) \\
& = - 1 + \frac{1}{b} + \log n - \frac{\Gamma'(b+1)}{\Gamma(b+1)} + \underset{n \rightarrow \infty}{o} (1) = - 1 + \log n - \frac{\Gamma'(b)}{\Gamma(b)} + \underset{n \rightarrow \infty}{o} (1), 
\end{align*}
where we have used the series representation of the Digamma function $\psi(\cdot):=\Gamma'(\cdot)/\Gamma(\cdot)$ at $b+1$ and the recurrence property $\psi(z+1)=\psi(z)+1/z$ of the Digamma function. This yields \eqref{equivlambdanabaequal12}. 
\end{proof}

\subsection{Quotients of Gamma functions, Digamma function} \label{fctzetaab}
In this appendix we gather some useful properties of the quotients of two Gamma functions and of the Digamma function. 

The following lemma is a useful estimate for the quotient of two Gamma functions, which can be found in, for example \cite{pjm/1102613160}. 
\begin{lemma} \label{equivquotientgamma}
For any $\alpha, \beta \in \R$ we have 
\begin{align*}
\frac{\Gamma(z+\alpha)}{\Gamma(z+\beta)} = z^{\alpha-\beta} \left ( 1 + \frac{(\alpha-\beta)(\alpha+\beta-1)}{2z} + \frac{C(\alpha,\beta)}{z^2} + \underset{z \rightarrow \infty}{o} \left (\frac1{z^2}\right ) \right ), 
\end{align*}
where $C(\alpha,\beta):=\frac1{12} \binom{\alpha-\beta}{2}(3(\alpha-\beta-1)^2-\alpha+\beta-1)$. 
\end{lemma}

\begin{lemma} \label{fctzetaablem}
Let $\alpha,\beta>0$ be such that $0<\alpha-\beta<1$ (resp. $\alpha-\beta<0$), then the function $x \mapsto \Gamma(x+\alpha)/\Gamma(x+\beta)$ is $\mathcal{C}^1$ and increasing (resp. decreasing) on $(-\alpha,\infty)$. It has limit $-\infty$ (resp. $\infty$) as $x \to -\alpha$ and $\infty$ (resp. $0$) as $x \to \infty$. It's inverse function is well-defined from $\mathbb{R}$ (resp. $(0,\infty)$) to $(-\alpha,\infty)$ and it is $\mathcal{C}^1$ and increasing (resp. decreasing). 
\end{lemma}

\begin{proof}
Let us denote by $f$ the function in the statement of the lemma. Smoothness of $f$ on $(-\alpha,\infty)\setminus\{-\beta\}$ comes from smoothness of the Gamma function and, for $x \in (-\alpha,\infty)\setminus\{-\beta\}$, an easy calculation yields 
\begin{align}
f'(x)=f(x)(\psi(x+\alpha)-\psi(x+\beta)), \label{exprfprime}
\end{align}
where $\psi(\cdot):=\Gamma'(\cdot)/\Gamma(\cdot)$ is the Digamma function. When $0<\alpha-\beta<1$, the value $f(-\beta)$ is defined by continuity and equals $0$ and, using \eqref{exprfprime} and that, as $z \to 0$, we have the asymptotics $\Gamma(z)\sim1/z$ and $\psi(z)\sim-1/z$, we obtain that $f'(x) \to \Gamma(\alpha-\beta)>0$ as $x \to -\beta$. By classical theorems we deduce that $f$ is $\mathcal{C}^1$ on $(-\alpha,\infty)$ even when $0<\alpha-\beta<1$. When $0<\alpha-\beta<1$ (resp. $\alpha-\beta<0$), for $x \in (-\beta,\infty)$ (resp. $x \in (-\alpha,\infty)$) we have $f(x)>0$ and, since the Digamma function is increasing on $(0,\infty)$, we deduce from \eqref{exprfprime} that $f'(x)>0$ (resp. $f'(x)<0$). When $0<\alpha-\beta<1$, for $x \in (-\alpha,-\beta)$, using the recurrence property of the Gamma function we get 
\begin{align}
\psi(x+\beta) = \psi(x+\beta+1) - \frac1{x+\beta} > \psi(x+\alpha), \label{ineqdigamma}
\end{align}
where, for the last inequality, we have used that $x+\beta+1>x+\alpha>0$, that the Digamma function is increasing on $(0,\infty)$, and that $x+\beta<0$. Combining \eqref{ineqdigamma} with the fact that $f(x)<0$ when $x \in (-\alpha,-\beta)$ and \eqref{exprfprime} we obtain that $f'(x)>0$ for $x \in (-\alpha,-\beta)$. Altogether, we obtain the claim about monotonicity on $(-\alpha,\infty)$. The limit at $-\alpha$ comes from the fact that $\Gamma(x+\alpha) \to \infty$ and $\Gamma(x+\beta) \to \Gamma(-\alpha+\beta)$ as $x \to -\alpha$. The limit at infinity comes from Lemma \ref{equivquotientgamma}. Since $f$ has non-vanishing derivative on $(-\alpha,\infty)$, its inverse is also $\mathcal{C}^1$ with non-vanishing derivative. 
\end{proof}

\begin{lemma} \label{convexity}
Let $\alpha,\beta>0$ be such that $0<\alpha-\beta<1$ (resp. $\alpha-\beta<0$), then the function $x \mapsto \Gamma(x+\alpha)/\Gamma(x+\beta)$ is concave (resp. convex) on $(-\alpha,\infty)$. Equivalently, it's inverse function is convex (resp. concave). 
\end{lemma}
It is possible to justify Lemma \ref{convexity} via analytical arguments, at least on $(-(\alpha \wedge \beta),\infty)$ (see for example \cite[Rem. 18]{qiluo}). However, we prefer to derive it as a corollary of Theorem \ref{asymptfunctional}. 
\begin{proof}
Let $k \geq 2$ and set $\psi_k(n):=\mathds{1}_{[k,\infty)}(n)$. If $0<\alpha-\beta<1$ (resp. $\alpha-\beta<0$), choosing $b=\alpha$ and $a=2+\beta-\alpha$, noting that we then have $a \in (1,2)$ (resp. $a>2$) and applying Theorem \ref{asymptfunctional} to $\psi_k$, we get from \eqref{asymptfunctional02} that, on $(-\infty,M^{\beta_{a,b}}(\psi_k))$, $\zeta_{a,b}(\cdot)$ is a limit of renormalized log-Laplace transforms so it is convex (see for example \cite[Lem. 2.3.9]{dembozeitouni}). Since $M^{\beta_{a,b}}(\psi_k)=\lambda_k(\beta_{a,b})$ and $k$ can be chosen arbitrary large, we get that $\zeta_{a,b}(\cdot)$ is convex on its whole domain. Then, the reciprocal of $x \mapsto \Gamma(x+\alpha)/\Gamma(x+\beta)$ is obtained by composing $\zeta_{a,b}(\cdot)$ and an affine function that has positive (resp. negative) derivative, so that reciprocal is convex (resp. concave). 
\end{proof}

The following lemma is classical so we omit its proof. 
\begin{lemma} \label{fctl2b}
The Digamma function $\psi(\cdot):=\Gamma'(\cdot)/\Gamma(\cdot)$ is smooth and concave-increasing on $(0,\infty)$. It has limit $-\infty$ as $x \to 0$ and $\infty$ as $x \to \infty$. It's inverse function is well-defined from $\mathbb{R}$ to $(0,\infty)$ and it is smooth and increasing. 
\end{lemma}

\begin{remark}
Proceeding as in the proof of Lemma \ref{convexity}, we see that the concavity of the Digamma function can also be recovered from Theorem \ref{asymptfunctional} (in the case $a=2$). 
\end{remark}

\subsection{The classical law of large numbers for absorption times} \label{appendlgn}

In this section we justify that, when $\Lambda=\beta_{a,b}$ with $a>1$ and $b>0$, the limit $\zeta'_{a,b}(0)$ from \eqref{lgn} indeed coincides with the limit $1/\mu(\beta_{a,b})$ from the classical result \eqref{lgnclassical}. On the one hand, since $\zeta'_{a,b}(0)=1/L'_{a,b}(0)$, using \eqref{deffcttoinvert} and \eqref{exprfprime} we get, when $a \in (1,\infty)\setminus\{2\}$, 
\begin{align}
\frac1{\zeta'_{a,b}(0)} = L'_{a,b}(0) & = \frac{\Gamma(a)}{(2-a)(a-1)} \frac{\dd}{\dd y} \left ( \frac{\Gamma(b+y)}{\Gamma(a+b+y-2)} \right )_{|y=0} \nonumber \\
& =  \frac{\Gamma(a) \Gamma(b) (\psi(b)-\psi(a+b-2))}{(2-a)(a-1) \Gamma(a+b-2)}, \label{limfromrecoveredlgnaneq2}
\end{align}
where $\psi(\cdot):=\Gamma'(\cdot)/\Gamma(\cdot)$ is the Digamma function. When $a=2$ we obtain similarly 
\begin{align}
\frac1{\zeta'_{2,b}(0)} = L'_{2,b}(0) & = \frac{\dd}{\dd y} \left ( \psi(b+y) \right )_{|y=0} = \psi'(b). \label{limfromrecoveredlgnaeq2}
\end{align}
On the other hand, using the definition of $\mu(\Lambda)$ in \eqref{lgnclassical}, that $\beta_{a,b}(dr):=r^{a-1}(1-r)^{b-1}\dd r$, the definition of the beta function $B(\cdot,\cdot)$, the identity $B(x,y)=\Gamma(x)\Gamma(y)/\Gamma(x+y)$ and \eqref{exprfprime}, we get, when $a \in (1,\infty)\setminus\{2\}$, 
\begin{align}
\mu(\beta_{a,b})&=-\int_0^1\log(1-r)r^{a-3}(1-r)^{b-1} \dd r = -\frac{\dd}{\dd y} \left ( \int_0^1 r^{a-3}(1-r)^{y-1} \dd r \right )_{|y=b} \nonumber \\
&= -\frac{\dd}{\dd y} \left ( B(a-2,y) \right )_{|y=b} = -\Gamma(a-2)\frac{\dd}{\dd y} \left ( \frac{\Gamma(y)}{\Gamma(y+a-2)} \right )_{|y=b} \nonumber \\
&= \frac{\Gamma(a)}{(2-a)(a-1)} \frac{\Gamma(b)}{\Gamma(a+b-2)} (\psi(b)-\psi(a+b-2)). \label{limclassicallgnaneq2}
\end{align}
Then, when $a=2$, using dominated convergence and \eqref{limclassicallgnaneq2} we get 
\begin{align}
\mu(\Lambda_{2,b})&=-\int_0^1\log(1-r)r^{-1}(1-r)^{b-1} \dd r = - \lim_{x \rightarrow 2} \int_0^1\log(1-r)r^{x-3}(1-r)^{b-1} \dd r \nonumber \\
&= \lim_{x \rightarrow 2} \frac{\Gamma(x)\Gamma(b)}{(x-1)\Gamma(x+b-2)} \frac{\psi(b+x-2)-\psi(b)}{x-2} = \psi'(b). \label{limclassicallgnaeq2}
\end{align}
Comparing \eqref{limfromrecoveredlgnaneq2} and \eqref{limclassicallgnaneq2} when $a \in (1,\infty)\setminus\{2\}$, and \eqref{limfromrecoveredlgnaeq2} and \eqref{limclassicallgnaeq2} when $a=2$, we get that, in any case, we have $\zeta'_{a,b}(0)=1/\mu(\beta_{a,b})$, proving the claim. 

\subsection*{Acknowledgment}
This work was supported by Beijing Natural Science Foundation, project number IS24067. 

\bibliographystyle{plain}
\bibliography{thbiblio}

\end{document}